\documentclass[article, twosided]{siamart171218}

\usepackage[utf8x]{inputenc}
\usepackage{textcomp}
\usepackage{mathtools,cases, csquotes}
\usepackage{amsmath, amsbsy}
\usepackage{cancel}
\usepackage{booktabs}
\usepackage{subcaption}
\usepackage{empheq}
\usepackage{enumerate}
\usepackage{listings}
\usepackage{afterpage}
\usepackage{enumitem} 
\usepackage{appendix}
\usepackage{xfrac}

\usepackage{lipsum}
\usepackage{amsfonts}
\usepackage{graphicx}
\usepackage{epstopdf}
\usepackage{algorithmic}
\ifpdf
  \DeclareGraphicsExtensions{.eps,.pdf,.png,.jpg}
\else
  \DeclareGraphicsExtensions{.eps}
\fi

\numberwithin{theorem}{section}

\newcommand{\TheTitle}{A Mixed Approach to the Poisson Problem with Line Sources} 
\newcommand{\TheAuthors}{I. G. Gjerde,K. Kumar, J. M. Nordbotten}

\headers{\TheTitle}{\TheAuthors}

\title{{\TheTitle} \thanks{This work has been supported by the Research Council of Norway, project number 250223. 
}}

\author{Ingeborg G. Gjerde\thanks{Department of Mathematics, University of Bergen, Norway. \email{(ingeborg.gjerde@uib.no)}}
\and
Kundan Kumar \thanks{Department of Mathematics, Karlstad University, Sweden.  \email{(kundan.kumar@karlstad.se)}}
\and Jan M. Nordbotten \thanks{Department of Mathematics, University of Bergen, Norway. \email{(jan.nordbotten@math.uib.no)}}}

\usepackage{amsopn}

\ifpdf
\hypersetup{
  pdftitle={\TheTitle},
  pdfauthor={\TheAuthors}
}
\fi

\newcommand{\norm}[2]{ \Vert #1 \Vert_{#2}}

\begin{document}

\maketitle

\begin{abstract}
In this work we consider the primal mixed variational formulation of the Poisson equation with a line source. 
The analysis and approximation of this problem is non-standard as the line source causes the solutions to be singular. We start by showing that this problem admits a solution in appropriately weighted Sobolev spaces. Next, we show that given some assumptions on the problem parameters, the solution admits a splitting into higher and lower regularity terms. The lower regularity terms are here explicitly known and capture the solution singularities. The higher regularity terms, meanwhile, are defined as the solution of its own mixed Poisson equation. With the solution splitting in hand, we then define a singularity removal based mixed finite element method in which only the higher regularity are approximated numerically. This method yields a significant improvement in the convergence rate when compared to approximating the full solution. In particular, we show that the singularity removal based method yields optimal convergence rates for lowest order Raviart-Thomas and discontinuous Lagrange elements.
\end{abstract}



\section{Introduction}
Let $\Omega \subset \mathbb{R}^3$ be a 3D domain with smooth boundary $\partial \Omega$. Let $\Lambda$ be a smooth 1D curve with the parametrization $\boldsymbol{\lambda} = [ \xi(s), \tau(s), \zeta(s)]$ so that $\Lambda =\{ \boldsymbol{\lambda}(s) \in (0,L) \} \subset \mathbb{R}^1 \subset \Omega$. For simplicity, we assume $\Vert \boldsymbol{\lambda}'(s) \Vert=1$ so that the arc-length and coordinate $s$ coincide. We consider in this work the mixed Poisson problem with a line source on $\Lambda$: Find $u$ and $\mathbf{q}$ solving
\begin{subequations}
\begin{align}
\mathbf{q} + \kappa \nabla u &= 0 && \text{ in } \Omega, \label{eq:darcy-strong} \\
\nabla \cdot \mathbf{q} &= f \,   \delta_\Lambda && \text{ in } \Omega, \label{eq:cons-strong}\\
u &= u_0 \label{eq:bc-strong} && \text{ on } \partial \Omega,
\end{align}
\end{subequations}
where $f \in C^0(\bar{\Omega})$ denotes the line source intensity, $\kappa \in L^{\infty}(\Omega)$ a strictly positive permeability,$u_0 \in C^2(\bar{\Omega})$ the boundary data and $\delta_\Lambda$ a Dirac line source concentrated on $\Lambda$. The line source $\delta_\Lambda$ is taken as the limit of a sequence of nascent Dirac functions of unit measure per arc length:
\begin{align}
\delta_\Lambda = \lim_{\epsilon\rightarrow0} \delta_\Lambda^\epsilon, \quad \delta_\Lambda^\epsilon = \begin{cases}  \frac{1}{\pi \epsilon^2} &\text{ for }  r\leq \epsilon, \\ 0 &\text{ otherwise},\end{cases} \label{eq:dirac-limit}
\end{align}
where $r = \text{dist}(\mathbf{x}, \Lambda)$.
\begin{figure}
	\centering
	\includegraphics[width=1.6in]{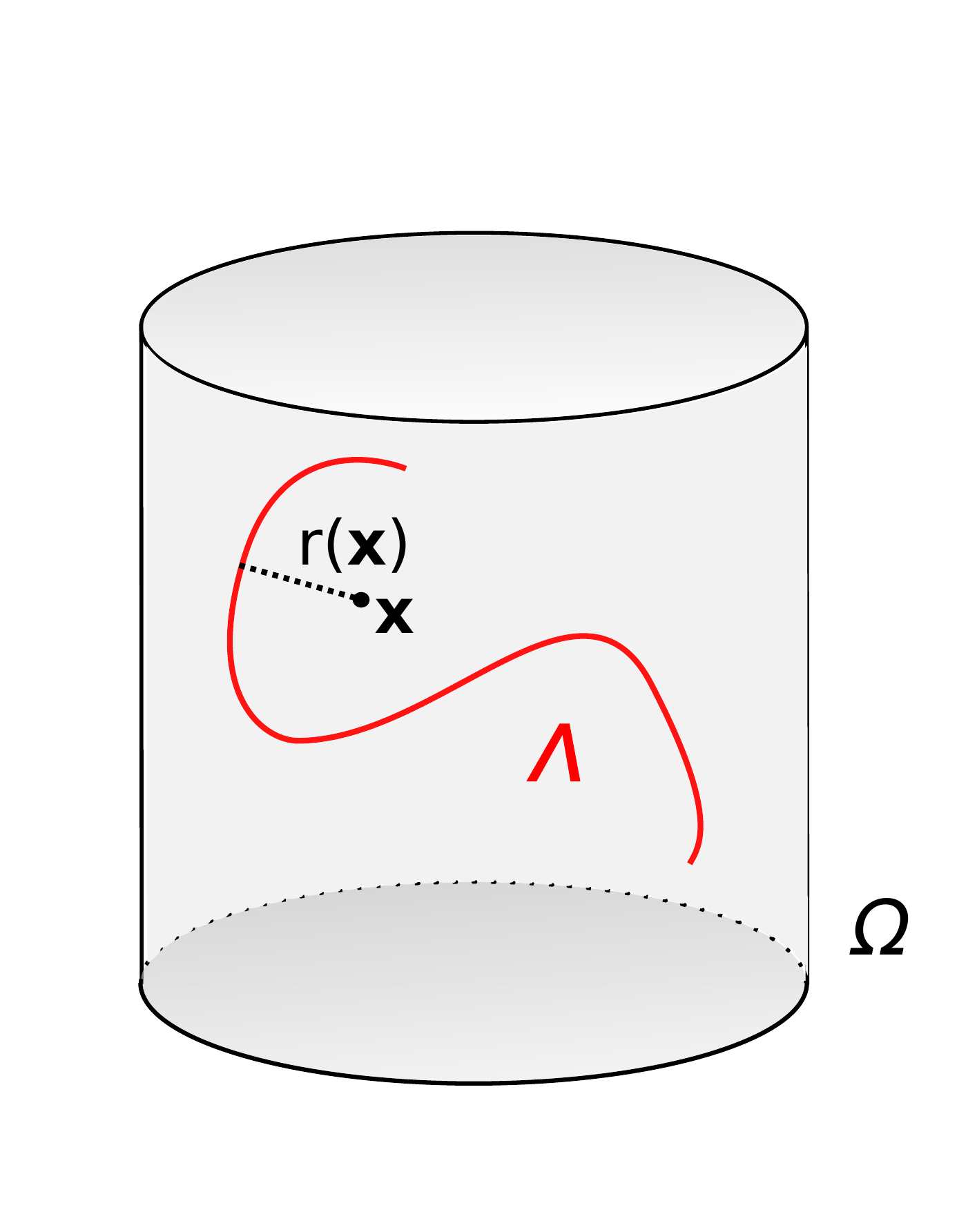}
	\caption{A 3D domain $\Omega \subset \mathbb{R}^3$ with an embedded line $\Lambda$, with $r(\mathbf{x})$ denoting the distance of a point in $\Omega$ to the line.}
	\label{fig:illustration}
\end{figure}

Models of the type \eqref{eq:darcy-strong}-\eqref{eq:bc-strong} arise in a variety of applications. In geophysics, line sources have been used to model 1D steel components in concrete  structures \cite{llau2016} and the interference of metallic pipelines and bore-casings in electromagnetic modelling of reservoirs \cite{weiss2017}. In the context of geothermal energy, line sources have been used to model the heat exchange between a well and the surrounding soil \cite{al-khoury2005}. In reservoir engineering, coupled 1D-3D flow models (where \eqref{eq:darcy-strong}-\eqref{eq:bc-strong} is coupled to a 1D flow equation on $\Lambda$) are used to model the flow between a well and reservoir \cite{Zunino-well, Wolfsteiner2003, aavatsmark2003index}. The same model has also been considered in the context of biological systems, where it has been used to model water flow through a root system \cite{root1,stuttgart-root}, blood and oxygen transport through the vascularized tissue of the brain \cite{Reichold2009, Grinberg2011, vidotto2018, Fang2008, Linninger2013}, the efficiency of cancer treatment by hyperthermia \cite{nabil2016}, and the efficiency of drug delivery through microcirculation \cite{Cattaneo2014, zunino2018}. 

In many of these applications, the flow equation \eqref{eq:darcy-strong}-\eqref{eq:bc-strong} will be coupled to a transport equation (describing for example heat flow or the concentration of some chemical). For this reason, we consider herein the mixed variational formulation of \eqref{eq:darcy-strong}-\eqref{eq:cons-strong}. Upon discretization, this will yield a mixed finite element method, which is known to provide good approximations of the velocity field. In particular, it provides locally conservative approximations. 

As we will see, the analysis and approximation of \eqref{eq:darcy-strong}-\eqref{eq:cons-strong} is non-standard as the line source $\delta_\Lambda$ induces the solution to be singular. To be more precise, one has that $\delta_\Lambda$ induces a logarithmic type singularity in $u$ and a $r^{-1}$-type singularity in $\mathbf{q}$. Consequently, one has $\mathbf{q} \, \cancel{\in} \, L^2(\Omega)$. The analysis of \eqref{eq:darcy-strong}-\eqref{eq:cons-strong} therefore requires non-standard Sobolev spaces. From a numerical perspective, the singular nature of $u$ and $\mathbf{q}$ make them challenging to approximate. 

In this work, we (i) prove the existence of a solution to \eqref{eq:darcy-strong}-\eqref{eq:cons-strong} and (ii) construct an efficient numerical method with which to approximate it. The existence of a solution is proved using suitably weighted Sobolev space, similar to the ones used in \cite{arbogast2016, westphal2008}. With these spaces, the proof follows by the generalized Lax-Milgram theorem together with a limit argument. As we will see, the analysis raises questions regarding the approximation properties of $\mathbf{q}$. For this reason, we extend our work from \cite{Gjerde2018} to show that with some assumptions on the problem parameters, the solution admits a splitting of the type
\begin{align}
u = u_s + u_r, \quad \mathbf{q} = \mathbf{q}_s + \mathbf{q}_r, \label{eq:splitting-intro}
\end{align}
where $u_s$ and $\mathbf{q}_s$ denote explicitly known terms capturing the solution singularities, and $u_r$ and $\mathbf{q}_r$ denote higher regularity remainder terms. The remainder terms are defined as the solution of their own mixed Poisson equation. With the splitting in hand, we then formulate a solution strategy in which only $u_r$ and $\mathbf{q}_r$ are approximated using a mixed finite element method. In contrast to the development in \cite{Gjerde2018}, this has the advantage of providing a locally mass conservative approximation. The full solution pair $(u, \mathbf{q})$ can then be reconstructed using \eqref{eq:splitting-intro}. We will refer to this as the singularity removal based mixed finite element method.  

\subsection{Relevant literature and our contribution}
Several authors have contributed to the analysis of the Poisson equation with a line source. Of special relevance to our work, we mention the work of D'Angelo and Quarteroni in \cite{dangelo2008}, where they proved the existence of a solution to the primal (non-mixed) variational formulation of \eqref{eq:darcy-strong}-\eqref{eq:bc-strong}. The proof relied on weighted Sobolev spaces similar to those known from the study of corner-point problems \cite{babuska1972}. In \cite{dangelo2012}, D'Angelo went on to study the finite element approximation of the problem. There, he found that the approximation converges sub-optimally in the $L^2$-norm, and fails to converge in the $H^1$-norm. Convergence can be improved by weighing the error-norm, and optimal convergence rates can be retrieved by grading the mesh, i.e. by performing a particular refinement around the singularity. A similar result is known for the point source problem in 2D \cite{Ding2001, apel2011}. Köppl et al. proved that the convergence issues are local to the singularity around the line source \cite{koppl2015}. For the coupled 1D-3D flow model, this means that the numerical approximation will suffer pollution until the mesh size $h$ is smaller than $R$ ($R$ being the original radius of e.g. blood vessel or well). We show in \cite{Gjerde2018-2} that the FE approximation of the coupled 1D-3D flow therefore requires a very fine mesh to converge. 

Several strategies have been proposed in order to deal with this issue regarding computational complexity. We refer to the work of Kuchta et al. in \cite{miro2016-2D1D, baerland2018} for suitable preconditioners for the coupled 1D-3D problem. Holter et al. then applied this preconditioner to simulate flow through the microcirculature found in a mouse brain \cite{Holter2018}. Koch et al. introduced a smoothing kernel to distribute the line source over a 3D subdomain \cite{Koch2019}. An alternative coupling scheme was introduced  by Köppl et al. in \cite{koppl2016, koppl2017}, where the source term was taken to live on the boundary of the inclusions. This idea was further developed in \cite{Zunino-well, Laurino2019}. The result is a 1D-(2D)-3D method where the dimensional gap has been reduced to 1, thus improving the approximation properties of the solution.

The work cited so far has all been on the primal variational formulation of \eqref{eq:darcy-strong}-\eqref{eq:bc-strong}. Comparatively little work has been done that considers its mixed formulation (as an exception, we note the work of Notaro et al. in providing a mixed FE discretization of the coupled 1D-3D flow model \cite{MFEM-linesource}). A mathematical analysis of \eqref{eq:darcy-strong}-\eqref{eq:bc-strong} is, however, to the best of our knowledge, still missing. As is the construction of a suitable numerical method with which to approximate the solution. The aim of this article is to fill this gap.

\subsection{Overview of the paper}
We start in Section \ref{sec:notation} by introducing the weighted Sobolev spaces. With these in hand, we then prove in Section \ref{sec:existence} the existence of a solution to the primal mixed variational formulation of \eqref{eq:darcy-strong}-\eqref{eq:bc-strong}. The solution is shown to exist in a highly non-standard space with poor approximation properties. For this reason, we proceed in Section \ref{sec:decomposition} to construct a solution splitting of the type \eqref{eq:splitting-intro}, where the solution is split into higher and lower regularity terms. In Section \ref{sec:disc} we give the mixed finite element discretization of the problem. Here, we provide two different methods: the standard mixed finite element method that approximates the full solution pair $(u, \mathbf{q})$, and a singularity removal based finite element method that approximates only the higher-regularity remainder pair $(u_r, \mathbf{q}_r)$. In Section \ref{sec:numerics-weighted}, we provide numerical evidence that the former method fails to converge
in the standard $L^2$-norm. In Section \ref{sec:numerics-reformulated}, we then show that the latter method, i.e., solving for the remainder pair $(u_r, \mathbf{q}_r)$, yields significantly improved convergence rates. We conclude by showing the 
the results of applying the singularity removal based mixed finite element method on a non-trivial geometry taken from the vascular network of a rat tumour.

\section{Function spaces and notation}
\label{sec:notation}
The purpose of this section is to introduce the
weighted Sobolev spaces in which solutions to \eqref{eq:darcy-strong}-\eqref{eq:bc-strong} belong. We start by giving the definition of the standard Sobolev spaces. Let $\mathrm{d}x$ denote the standard Lebesgue measure in $\mathbb{R}^3$, $\sigma$ the $\sigma$-algebra on $\Omega$ and $(\Omega, \sigma, \mathrm{d}x)$ the usual Lebesgue measure space. Letting $L^2(\Omega)$ denote the space of square integrable functions on $(\Omega, \sigma, \mathrm{d}x)$, the Sobolev space $H^m(\Omega)$ can be defined as
\begin{align*}
H^m(\Omega) = \lbrace D^\beta u \in L^2(\Omega) \text{ for all } \, \vert \beta \vert \leq m  \rbrace,
\end{align*}
equipped with the inner product
\begin{align*}
(u,v)_{ H^m(\Omega)} = \sum_{\vert \beta \vert \leq m} (D^\beta u, D^\beta v),
\end{align*}
where $(\cdot, \cdot)$ denotes the $L^2$-inner product $(u, v)_\Omega = \int_\Omega uv \, \mathrm{d}x$, $\beta$ is a multi-index and $D^\beta u$ denotes the corresponding distributional partial derivative of $u$. A subscript $H^m_0(\Omega)$ is used to denote the subspace of $H^m(\Omega)$ with zero trace on the boundary. Next, let $H(\mathrm{div}; \Omega)$ be given as
\begin{align*}
H(\text{div}; \Omega) = \{ \mathbf{q} \in (L^2(\Omega))^3 : \nabla \cdot \mathbf{q} \in L^2(\Omega)\},
\end{align*}
equipped with inner product
\begin{align*}
(\mathbf{q}, \mathbf{v})_{H(\text{div};\Omega)} = (\mathbf{q}, \mathbf{v}) + (\nabla \cdot \mathbf{q}, \nabla \cdot \mathbf{v}).
\end{align*} 

Let $r(\mathbf{x}) = \text{dist}(\mathbf{x}, \Lambda)$ denote the distance of a point $\mathbf{x} \in \Omega$ to $\Lambda$. As we will see in Section \ref{sec:decomposition}, the line source $\delta_\Lambda$ introduces a $r^{-1}$-type singularity in $\mathbf{q}$. For this reason, $\mathbf{q}$ fails to belong to $L^2(\Omega)$; consequently, it also fails to belong to the standard $H(\mathrm{div}; \Omega)$-space. As we will see, the solution $\mathbf{q}$ will instead belong to a weighted $H$-div space. Let $\alpha \in \mathbb{R}$ and take $L^2_\alpha(\Omega)$ to denote the weighted space
\begin{align*}
L_\alpha^2(\Omega) := \left \{u \text{ measurable }: \int_\Omega \left( r^\alpha u \right)^2 \mathrm{d} x < \infty \right \}.
\end{align*}
This is a Hilbert space equipped with the scalar product 
\begin{align*}
(u, v)_{L_\alpha^2(\Omega)} = \int_\Omega r^{2\alpha} u v \, \mathrm{d} x.
\end{align*}
\textcolor{black}{Formally, the value of $\alpha$ controls how singular the function is allowed to be. Increasing $\alpha$ leads to an increase in the space $L^2_\alpha(\Omega)$; i.e. 
letting $\alpha_1<\alpha_2$, one has $L^2_{\alpha_1}(\Omega) \subset L^2_{\alpha_2}(\Omega)$.} Next, by an application of Cauchy-Schwarz, we obtain 
\begin{align}
\vert (u, v) \vert = \vert (r^\alpha u, r^{-\alpha}v) \vert \leq \norm{u}{L^2_{\alpha}(\Omega)}  \norm{v}{L^2_{-\alpha}(\Omega)}  \quad \forall \, u \in L^2_\alpha(\Omega), \, v \in L^2_{-\alpha}(\Omega),
\label{eq:weighted-cauchy-schwarz}
\end{align}
\textcolor{black}{meaning that the spaces $L^2_\alpha(\Omega)$ and $L^2_{-\alpha}(\Omega)$ are dual to each other.}

For $\alpha \in (-1,1)$, the weights $r^\alpha$ are said to be \textit{Muckenhoupt}, and we have the imbedding $L^2_{\alpha}(\Omega) \hookrightarrow L^1(\Omega)$ \cite{Turesson2000}. The space $L^2_\alpha(\mathrm{div}; \Omega)$ then admits properties such as density of smooth functions $C^\infty_0(\Omega; \mathrm{d}x)$. For general $\alpha$, the properties of $L^2_\alpha$ are best understood in the context of measure theory. Let $\mathrm{d}\mu = r(x)^\alpha \mathrm{d} x$; this defines a measure for all $\alpha \in \mathbb{R}$ and the triple $(\Omega, \sigma, \mathrm{d} \mu)$ constitutes a measure space. The space $L^2_\alpha(\Omega)$ can then equivalently be defined as the $L^2$-space on $(\Omega, \sigma, \mathrm{d} \mu)$:
\begin{align*}
L^2(\Omega; \mathrm{d} \mu) = \{ u \text{ measurable } : \int_\Omega u^2 \mathrm{d} \mu  < \infty \}.
\end{align*} 
Thus, $L^2_\alpha(\Omega)$ admits the standard properties of $L^2$-spaces with respect to $(\Omega, \sigma, \mathrm{d} \mu)$, such as density of smooth functions $C^\infty_0(\Omega; \mathrm{d}\mu)$.
In particular, it is complete \cite[Thm. 13.11]{real-analysis}. 

Let $H_{\alpha}^1(\Omega)$ denote the space \cite{kufner, koslov, oleinik, kilpelainen1994}:
\begin{align*}
H_{\alpha}^1(\Omega) = \lbrace u \in L^2_\alpha(\Omega) : \nabla u \in (L^2_{\alpha}(\Omega))^3 \rbrace. 
\end{align*}
This is a Hilbert space equipped with inner product
\begin{align*}
(u,v)_{ H_{\alpha}^1(\Omega)} = (u,v)_{L^2_\alpha(\Omega)} + (\nabla u,\nabla v)_{L^2_{\alpha}(\Omega)}.
\end{align*}
and is a Sobolev space in the sense that $r^\alpha u \in H^1(\Omega)$. The space $H^1_\alpha(\Omega)$ is often referred to as a non-homogeneous weighted Sobolev space, as the weight is not adjusted to compensate for the regularity lost when taking a derivative. In this work, we shall work mainly with homogeneous weighted Sobolev spaces of the type
\begin{align*}
V_{\alpha}^1(\Omega) = \lbrace u \in L^2_{\alpha-1}(\Omega) : \nabla u \in (L^2_{\alpha}(\Omega))^3 \rbrace,
\end{align*}
The $H_{\alpha}^1(\Omega)$ and $V_{\alpha}^1(\Omega)$ norms are equivalent; this follows from the following inequality \cite{babuska1972}:
\begin{align}
\norm{u}{L^2_{\alpha-1}(\Omega)} \leq C_\alpha \norm{u}{H^1_{\alpha}(\Omega)}, \label{eq:imbedding}
\end{align}
The properties of the spaces $H_{\alpha}^1(\Omega)$ and $V_{\alpha}^1(\Omega)$ depend on the choice of weights. For $\alpha \in (-1,1)$, the weights used in the space $H_{\alpha}^1(\Omega)$ are both Muckenhoupt. One then has density of smooth functions and the imbedding $L^1(\Omega) \subset H^1_\alpha(\Omega)$. By equivalence of norms, the same holds for the space $V_{\alpha}(\Omega)$.

Finally, let us define the weighted $H$-div type space $V_{\alpha+1}(\mathrm{div}; \Omega)$:
\begin{align*}
V_{\alpha+1}(\text{div}; \Omega) = \lbrace \mathbf{q} \in (L^2_\alpha(\Omega))^3 : \nabla \cdot \mathbf{q} \in L^2_{\alpha+1}(\Omega) \rbrace. 
\end{align*}
This is a Hilbert space equipped with the inner product
\begin{align*}
(\mathbf{q}, \mathbf{v})_{ V_{\alpha+1}(\Omega; \text{div})} = (\mathbf{q},\mathbf{v})_{L^2_\alpha(\Omega)} + (\nabla \cdot \mathbf{q},\nabla \cdot \mathbf{v})_{L^2_{\alpha+1}(\Omega)}.
\end{align*}
Note that elements of this space have a weak divergence $\nabla \cdot \mathbf{q} \in L_{\alpha+1}(\Omega)$, which is non-Muckenhoupt for $\alpha>0$. Consequently, the weak divergence of functions in $V_{\alpha+1}(\mathrm{div}; \Omega)$ may not belong to $L^1(\Omega)$.

\section{Existence of a Solution}
\label{sec:existence}
In the previous section, we gave the definition of the weighted Sobolev spaces. With these at our disposal, we are now ready to give the variational formulation of \eqref{eq:darcy-strong}-\eqref{eq:bc-strong}:
Find $(u, \mathbf{q}) \in L^2_{\alpha-1}(\Omega) \times V_{\alpha+1}(\mathrm{div}; \Omega) $ such that
\begin{subequations}
\begin{align}
 (\kappa^{-1} \mathbf{q} , \mathbf{v}) - ( u, \nabla \cdot \mathbf{v}) + (u_0, \mathbf{v}\cdot \mathbf{n})_{\partial \Omega} &= 0     &&  \forall \, \mathbf{v} \in V_{-\alpha+1}(\mathrm{div}; \Omega), \label{eq:darcy}\\
  (\nabla \cdot \mathbf{q}, \theta) &= (f \delta_\Lambda, \theta) && \forall \,  \theta \in L^2_{-\alpha-1}(\Omega), \label{eq:cons} 
\end{align}
\end{subequations} 
where $\mathbf{n}$ is the boundary unit normal of $\partial \Omega$. The solution space is chosen so that $r^{\alpha-1} p \in L^2(\Omega)$, $r^{\alpha} \mathbf{q} \in (L^2(\Omega)^3)$ and $r^{\alpha+1} \nabla\cdot \mathbf{q} \in L^2(\Omega)$, where the weighing is increased to account for the regularity loss caused by taking a derivative. This ensures that the velocity and pressure  spaces are sufficiently large to capture the expected structure of the solution, while selecting the largest test spaces admissible with respect to the bilinear forms appearing in the variational formulation. 
The main result of this section is the following existence theorem:
\begin{theorem}
\label{thm:primal}
Let $\Omega \subset \mathbb{R}^3$ be an open 3D domain with smooth boundary $\partial \Omega$, $\Lambda =\cup_{i=1}^n \Lambda_i$ be a collection of smooth, finite-length 1D curves $\Lambda_i \subset \mathbb{R}^1 \subset \Omega$, boundary data $u_0 \in C^2(\bar{\Omega})$, $f \in C^0(\bar{\Omega})$ and $\kappa \in L^{\infty}(\Omega)$ be strictly positive. For $\alpha>0$, there then exists $(u, \mathbf{q}) \in L^2_{\alpha-1}(\Omega) \times V_{\alpha+1}(\mathrm{div}; \Omega)$ solving \eqref{eq:darcy}-\eqref{eq:cons}.
\end{theorem}

The proof of Theorem \ref{thm:primal} relies on two lemmas. The first of these guarantees a solution to \eqref{eq:darcy}-\eqref{eq:cons} for a source term $g \in L^2_{\alpha+1}(\Omega)$ for general $\alpha \in \mathbb{R}$: 
\begin{lemma} 
\label{lemma:primal}
Let $g \in L^2_{\alpha+1}(\Omega)$. Under the assumptions of \ref{thm:primal}, there then exists $(u, \mathbf{q}) \in L^2_{\alpha-1}(\Omega) \times V_{\alpha+1}(\mathrm{div}; \Omega)$ solving
\begin{subequations}
\begin{align}
(\kappa^{-1} \mathbf{q}, \mathbf{v}) - (\nabla \cdot \mathbf{v}, u) + (u_0, \mathbf{v}\cdot \mathbf{n})_{\partial \Omega}=& \, 0 && \forall \, \mathbf{v}  \in V_{-\alpha+1}(\mathrm{div}; \Omega), \label{eq:mixed-lemma-proof1}\\
(\nabla \cdot \mathbf{q}, \theta) =& \, (g, \theta) && \forall \, \theta \in L^2_{-\alpha-1}(\Omega).\label{eq:mixed-lemma-proof2}
\end{align}
\end{subequations}
\end{lemma}

The second lemma addresses the line source:
\begin{lemma}
For $\alpha>0$ and $\delta_\Lambda$ in \eqref{eq:dirac-limit}, one has $\delta_\Lambda \in L^2_{\alpha+1}(\Omega)$. 
\label{lemma:dirac}
\end{lemma}
This section will proceed as follows. First, we give a definition of weak coercivity. With this in hand, we then state the Brezzi–Nečas–Babuška (BNB) Theorem \cite[Thm 2.6]{bernardi1988} (sometimes referred to as the Generalized Lax-Milgram theorem). After this, we give a proof of Lemma \ref{lemma:primal}; this is done by verifying the assumptions of the BNB Theorem. Next, we give a proof of Lemma \ref{lemma:dirac}. This is done by showing that the sequence $\delta_\Lambda^\epsilon$ is Cauchy in $L^2_{\alpha+1}(\Omega)$ (which is complete) and thus converges in $L^2_{\alpha+1}(\Omega)$. It follows that the line source $\delta_\Lambda$ belongs to $L^2_{\alpha+1}(\Omega)$. We conclude by giving a proof of \ref{thm:primal}.

\begin{theorem}[BNB Theorem]
\label{sec:bnb-theorem}
Let $X_i$ and $M_i$ be real reflexive Banach spaces ($i=1,2$). Assume we are given three continuous bilinear forms: $a: X_2 \times X_1 \rightarrow \mathbb{R}, \, b_1: X_1 \times M_1 \rightarrow \mathbb{R}, \, b_2: X_2 \times M_2 \rightarrow \mathbb{R}$. For any given $f \in (M_2)^*$ and $g \in (X_1)^*$, we consider the following problem:

Find $(q, u) \in X_2 \times M_1$ s.t.
\begin{subequations}
\begin{align}
a(q, v) + b_1(v, u) &= \langle g, v \rangle, \label{eq:bnb-problem1}\\
b_2(q, \theta) &= L(\theta)
\label{eq:bnb-problem2}
\end{align}
\end{subequations}
for all $(v, \theta) \in X_1 \times M_2$.

Let $K_i$ denote the kernel space of $b_i$:
\begin{align*}
K_i &= \lbrace  v \in X_i : b_i(v, u) = 0 \quad \forall \, u \in M_i  \rbrace.
\end{align*}
The problem \eqref{eq:bnb-problem1}-\eqref{eq:bnb-problem2} then admits a solution $(\mathbf{q}, u) \in X_2 \times M_1$ if the following assumptions hold:
\hspace{-5em} 
\begin{itemize}[itemindent=5em]
\item[Condition $(C_0)$:] Weak coercivity of $a(\cdot, \cdot)$: There exists constants $\gamma_1, \gamma_2>0$ s.t. 
\begin{align}
\sup_{v \in K_1} \frac{a(q,v)}{\norm{v}{X_1}} \geq \gamma_1 \norm{q}{X_2} \quad \forall \, q \in K_2, \label{eq:weak-coercivity1}
\end{align}
and 
\begin{align}
\sup_{q \in K_2} \frac{a(q,v)}{\norm{q}{X_1}} \geq \gamma_2 \norm{v}{X_1} \quad \forall \, v \in K_1. \label{eq:weak-coercivity2}
\end{align}
\item[Condition $(C_i):$] Inf-sup condition on $b_i(\cdot, \cdot)$
There exists $\beta_i>0$ s.t. 
\begin{align}
\sup_{v \in X_i} \frac{b_i(v, u)}{\norm{v}{X_i}} \geq \beta_i \norm{p}{M_i} \quad \forall \, u \in M_i. \label{eq:infsup}
\end{align}
\end{itemize}
\end{theorem}

\begin{proof}[Proof of Lemma \ref{lemma:primal}]
The proof consists of verifying the conditions of the BNB-theorem, taking the variational forms 
\begin{align*}
a(\mathbf{q}, \mathbf{v}) &= (\kappa^{-1}\mathbf{q}, \mathbf{v}), \\
 b_1(\mathbf{q}, \theta) = b_2(\mathbf{q}, \theta) &= -(\nabla \cdot \mathbf{q}, \theta), \\ 
\langle g, v \rangle &=  (u_0, \mathbf{v} \cdot \mathbf{n})_{\partial \Omega},
\end{align*}
and the function spaces $X_2 = V_{\alpha+1}(\mathrm{div}; \Omega)$, $X_1 = V_{-\alpha+1}(\mathrm{div}; \Omega), M_1 = L^2_{\alpha-1}(\Omega)$ and $M_2 = L^2_{-\alpha-1}(\Omega)$. The kernel spaces $K_1$ and $K_2$ are then given as 
\begin{align*}
K_1 &= \lbrace  \mathbf{v} \in V_{-\alpha+1}(\mathrm{div}; \Omega) : b_1(\mathbf{v}, u) = 0 \quad \forall \, u \in L^2_{\alpha-1}(\Omega)  \rbrace, \\
K_2 &= \lbrace  \mathbf{q} \in V_{\alpha+1}(\mathrm{div}; \Omega) : b_2(\mathbf{q}, \theta) = 0 \quad \forall \, \theta \in L^2_{-\alpha-1}(\Omega)  \rbrace.
\end{align*}

By an application of Cauchy-Schwarz \eqref{eq:weighted-cauchy-schwarz}, it then follows that the bilinear form $a(\cdot, \cdot)$ is bounded on $X_2 \times X_1$. The same holds true for $b(\cdot, \cdot)$ on $X_2 \times M_2$ and $X_1 \times M_1$, and $\langle g, v \rangle$ on $X_1$.

Next, let us consider equation \eqref{eq:weak-coercivity1} in Condition $(C_0)$. Before proceeding, let us first note a central property of the kernel space $K_2$. Fix arbitrary $\mathbf{q} \in K_2$ and take $\theta = r^{2(\alpha+1)} \nabla \cdot \mathbf{q} \in M_2$. This yields $b_2(\mathbf{q}, r^{2(\alpha+1)} \nabla \cdot \mathbf{q})=\norm{\nabla \cdot \mathbf{q}}{L^2_{\alpha+1}(\Omega)}=0$ for all $\mathbf{q} \in K_2$.
This yields
\begin{align*}
a(\mathbf{q}, r^{2 \alpha} \mathbf{q})= \norm{\kappa^{-1}\mathbf{q}}{L^2_\alpha(\Omega)}^2 \geq C_\kappa \norm{\mathbf{q}}{L^2_\alpha(\Omega)}^2 +  \norm{\nabla \cdot \mathbf{q}}{L^2_{\alpha+1}(\Omega)}^2 \geq C \norm{\mathbf{q}}{X_2}^2
\end{align*}
where we used that $\kappa \in L^\infty(\Omega)$. Next, a computation shows
\begin{align*}
\norm{\mathbf{v}}{X_1}^2 &= \norm{r^{2\alpha} \mathbf{q}}{X_1}^2 = \norm{r^{2\alpha }\mathbf{q}}{L^2_{-\alpha}(\Omega)}^2 + \norm{\nabla \cdot (r^{2\alpha }\mathbf{q})}{L^2_{-\alpha+1}(\Omega)}^2  \\
&\leq \norm{\mathbf{q}}{L^2_{\alpha}(\Omega)}^2 + 2 \alpha \norm{r^{2\alpha-1 } \nabla r \cdot \mathbf{q}}{L^2_{-\alpha+1}(\Omega)}^2 + \norm{r^{2\alpha} \nabla \cdot \mathbf{q}}{L^2_{-\alpha+1}(\Omega)}^2  \\
& \leq \norm{\mathbf{q}}{X_2}^2 + 2 \alpha \norm{ \nabla r \cdot \mathbf{q}}{L^2_{\alpha}(\Omega)}^2. 
\end{align*}
As $\nabla r \in C^\infty(\Omega)$, it follows that
\begin{align*}
\sup_{\mathbf{v} \in K_1} \frac{a(\mathbf{q}, \mathbf{v})}{\norm{\mathbf{v}}{X_1}} \geq \gamma \norm{\mathbf{q}}{X_2} \quad \forall \, \mathbf{q} \in K_2
\end{align*}
and \eqref{eq:weak-coercivity1} holds for some $\gamma_1 > 0$. To show that \eqref{eq:weak-coercivity2} holds, one can switch the sign of $\alpha$ and repeat this argument with $(\theta,\mathbf{q})$ switched with $(u,\mathbf{v})$. It follows that Condition $(C_0)$ holds.

Next we will verify Condition $(C_{i})$. Consider first $i=1$. The proof works by constructing a suitable $\mathbf{v} \in X_1$ for each $u \in M_{1}$. Let $\mathbf{v}$ solve the equation $\nabla \cdot \mathbf{v} = r^{2\alpha-2}u \in M_{1-\alpha}$. Further setting $\mathbf{v} = \nabla \xi$ then requires solving the Poisson problem $\Delta \xi = r^{2\alpha-2}u$ where the right-hand side $r^{2\alpha-2}u$ belongs to $L^2_{\alpha-1}$.
From \cite[Theorem 1]{oleinik} we know there exists such a solution $\xi \in H^2_{\alpha-1}(\Omega)$; thus a solution $\mathbf{v} \in H^1_{\alpha-1}(\mathrm{div};\Omega)$ exists solving to $\nabla \cdot \mathbf{v} = r^{2\alpha-2}u$. For this $\mathbf{v}$, $b(\mathbf{v}, u)=\int_\Omega (r^{\alpha-1}u)^2 \mathrm{d}\Omega = \norm{u}{M_{1}}^2$. Furthermore, 
\begin{align*}
\norm{\mathbf{v}}{X_1}^2 &= \norm{r^{-\alpha}\mathbf{v}}{L^2(\Omega)}^2 + \norm{r^{-\alpha+1}\nabla \cdot \mathbf{v}}{L^2(\Omega)}^2  \\
&= \norm{r^{-\alpha}\mathbf{v}}{L^2(\Omega)}^2 + \norm{r^{-\alpha+1} r^{2\alpha-2}u}{L^2(\Omega)}^2 \\
&= \norm{r^{-\alpha}\mathbf{v}}{L^2(\Omega)}^2 + \norm{r^{\alpha-1}u}{L^2(\Omega)}^2 \\
&= \norm{r^{-\alpha}\mathbf{v}}{L^2(\Omega)}^2 + \norm{u}{M_1}^2
\end{align*}
and it is obvious that $\norm{\mathbf{v}}{X_{2}}^2 \geq  \norm{u}{M_1}^2$. It follows that
\begin{align*}
\frac{b_1(\mathbf{v}, u)}{\norm{\mathbf{v}}{X_{2}} \norm{u}{M_1}}  = \frac{\norm{u}{M_{1}}^2}{ \norm{\mathbf{v}}{X_{2}}\norm{u}{M_1}}  \geq 1 
\end{align*}
and \eqref{eq:infsup} holds with $\beta_i \leq 1$. To show \eqref{eq:infsup} for $i=2$, one can switch the sign of $\alpha$ and repeat this argument with $(\theta,\mathbf{q})$ switched with $(u,\mathbf{v})$. It follows that Condition $(C_{2})$ holds. 
\end{proof}

\begin{figure}
	\centering
	\includegraphics[width=2.0in]{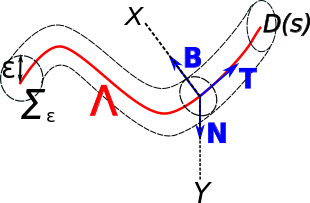}
	\caption{A generalized cylinder $\Sigma_\epsilon$ with centreline $\Lambda$ and a constant radius $\epsilon$. The curve $\Lambda$ is associated with a Frenet frame $\mathbf{T}, \mathbf{N}, \mathbf{B}$; here, $\mathbf{T}$ denotes its unit tangent vector, $\mathbf{N}$ its unit normal vector, and $\mathbf{B}$ its unit binormal vector.}
	\label{fig:frenet}
\end{figure}

\begin{proof}[Proof of Lemma \ref{lemma:dirac}]
First, let us rewrite $\delta_\Lambda$ to an equivalent definition:
\begin{align}
\delta_\Lambda = \lim_{k\rightarrow \infty} \delta_\Lambda^k, \quad \delta_\Lambda^k = \begin{cases} \frac{k^2}{\pi} & \text{ for } r \leq \frac{1}{k}  \\ 0 &  \text{ otherwise} \end{cases}
\end{align}
The proof is by showing that $f_k$ is a Cauchy sequence in $L^2_{\alpha+1}(\Omega)$ for $\alpha>0$. 

For each $k \in \mathbb{R}$, the function $f_k$ can be interpreted as the indicator function of a generalized cylinder $\Sigma$ with centreline $\Lambda$ and a constant radius $1/k$. Using then the notation of generalized cylinders \cite{gen-cyl}, we let $\mathbf{T}, \mathbf{N}, \mathbf{B}$ be the Frenet frame of $\Lambda$, as illustrated in Figure \ref{fig:frenet}. We further let $X$ and $Y$ denote the axes along the vectors $\mathbf{N}, \mathbf{B}$ of the Frenet frame; the coordinate axes $X,Y$ thus form a local coordinate system having origin on $\Lambda$. With this notation in hand, we have
\begin{align*}
\norm{f_{k+m}-f_k}{L^2_{\alpha+1}(\Omega)}^2 &= \int_\Sigma  (f_{k+m}-f_k)^2 r^{2\alpha+2}  \mathrm{d}
\Sigma \\
&= \int_\Lambda \int_{D(s)} (f_{k+m}-f_k)^2 r^{2\alpha+2}  \mathrm{d}
X(s) \mathrm{d} Y(s) \mathrm{d}s \\
&=  \int_0^{L } \underbrace{ \int_0^{2\pi} \int_0^{\frac{1}{k}}  (f_{k+m}-f_k)^2 r^{2\alpha+2} r \mathrm{d}r(s) \mathrm{d}\theta(s)}_{I(r,\theta;s)} \mathrm{d}s,
\end{align*}
where $D(s)$ denotes some parametrization of the cross section of $\Sigma$ at the point $\boldsymbol{\lambda}(s)$, $L$ denotes the length of $\Lambda$, and $r(s), \theta(s)$ the cylindrical coordinates of $X,Y$.

Consider for a moment $s$ to be fixed. A calculation then shows
\begin{align*}
I(r,\theta; s)&=\int_0^{2\pi} \int_0^{\frac{1}{k}}  (f_{k+m}-f_k)^2 r^{2\alpha+2} r \mathrm{d}r \mathrm{d}\theta \\ 
&= \frac{2\pi}{\pi^2} \left( \int_0^{\frac{1}{k+m}} ((k+m)^2-k^2)^2 r^{2\alpha+3}  \mathrm{d}r +  \int_{\frac{1}{k+m}}^{\frac{1}{k}} k^4 r^{2\alpha+3}  \mathrm{d}r \right) \\
&= \frac{2}{\pi}  \left( k^{-2\alpha} +  (m^2(2k+m)^2 - k^4 )(k+m)^{-4-2\alpha} \right).
\end{align*}
Note now that each term in the last line has a negative exponent, meaning that each term has a zero limit as $k,m \rightarrow \infty$. It follows that 
\begin{align}
\lim_{k,m \rightarrow \infty} I(r, \theta; s) = 0,
\end{align}
and consequently
\begin{align}
\lim_{k,m \rightarrow \infty} \norm{f_{k+m}-f_k}{L^2_{\alpha+1}(\Omega)}^2 = 0.
\end{align}
The sequence is thus Cauchy. Moreover, as the space $L^2_{\alpha+1}(\Omega)$ is complete, it follows that
 \begin{align}
\lim_{k \rightarrow \infty} f_{k} = f \in L^2_{\alpha+1}(\Omega).
\end{align}
\end{proof}

Finally, with Lemmas \ref{lemma:primal} and \ref{lemma:dirac} in hand, the proof of \ref{thm:primal} is straightforward:
\begin{proof}[Proof of \ref{thm:primal}]
By Lemma \ref{lemma:primal}, the problem \eqref{eq:darcy}-\eqref{eq:cons} satisfies all the conditions of the BNB-theorem except for boundedness of the right-hand side $(f, \theta)_\Lambda$. By Lemma \ref{lemma:dirac}, we know however that  $\delta_\Lambda$ as defined by \eqref{eq:dirac-limit} belongs to $L^2_{\alpha+1}(\Omega)$. Consequently, one has by Cauchy-Schwarz 
\begin{gather}
\begin{aligned}
(f \delta_\Lambda, \theta)_\Omega & \leq \norm{f}{L^\infty(\Omega)} \norm{\delta_\Lambda}{L^2_{\alpha+1}(\Omega)} \norm{\theta}{L^2_{-\alpha-1}(\Omega)}.
\end{aligned}
\end{gather}
\end{proof}

\section{Solution Splitting}
\label{sec:decomposition}
In the previous section, we proved the existence of $(u, \mathbf{q}) \in L^2_{\alpha-1}(\Omega) \times V_{\alpha+1}(\mathrm{div}; \Omega)$ solving \eqref{eq:darcy}-\eqref{eq:cons}. As was discussed in Section \ref{sec:notation}, the space $V_{\alpha+1}(\mathrm{div}; \Omega)$ is not Muckenhoupt. Consequently, $ V_{\alpha+1}(\mathrm{div}; \Omega) \cancel{\subset} L^1(\Omega) $. This leaves the approximation properties of $V_{\alpha+1}(\mathrm{div}; \Omega)$ highly non-standard. 

In this section, we will construct a solution splitting that can later be used to define a singularity removal based method for approximating $(u,\mathbf{q})$. Let $\Lambda =\cup_{i=1}^n \Lambda_i$ be a collection of line segments $\Lambda_i\subset \Omega$, $f \in C^0(\bar{\Omega}) \cap H^2(\Omega)$ and $\kappa>0$ be constant. Note that the regularity requirement on $f$ could be relaxed to taking $f$ piecewise $H^2$ on $\Lambda_i$ as in \cite[Section 3]{Gjerde2018}. Similarly, $\kappa$ could be taken scalar-valued using the splitting shown in \cite[Section 3.3]{Gjerde2018}. Considering again the strong formulation of the problem,
\begin{subequations}
\begin{align}
\mathbf{q} + \kappa \nabla u &= 0 && \text{ in } \Omega, \label{eq:darcy-splitting} \\
\nabla \cdot \mathbf{q} &= f \delta_\Lambda && \text{ in } \Omega, \label{eq:cons-splitting}\\
u &= u_0 \label{eq:bc-splitting} && \text{ on } \partial \Omega,
\end{align}
\end{subequations}
the solution admits a splitting into higher and lower-regularity terms
\begin{align}
(u, \mathbf{q}) = (u_s, \mathbf{q}_s) + (u_r, \mathbf{q}_r) \text{ where } \begin{cases} 
(u_s, \mathbf{q}_s) &\in L^2_{\alpha-1}(\Omega) \times V_{\alpha+1}(\mathrm{div}; \Omega),  \\
(u_r, \mathbf{q}_r) &\in L^2(\Omega) \times H(\mathrm{div}; \Omega).
\end{cases} \label{eq:sol-splitting}
\end{align}
Here, the lower-regularity pair $(u_s, \mathbf{q}_s)$ is defined as
\begin{align}
u_s := fG, \quad \mathbf{q}_s := - \kappa \nabla u_s,
\end{align}
with $G$ taken as the solution of $- \kappa \Delta G = \delta_{\Lambda}$ in $\mathbb{R}^3$ in an appropriately weak sense. In the next section we will show that this property ensures that $(u_s, \mathbf{q}_s)$ capture the singular behaviour induced by $\delta_{\Lambda}$. This allows the remainder pair $(u_r, \mathbf{q}_r)$ to enjoy higher regularity. Inserting \ref{eq:sol-splitting} into \eqref{eq:darcy-splitting}-\eqref{eq:bc-splitting}, and using $- \kappa \Delta G = \delta_{\Lambda}$, one finds that the remainder pair must satisfy
\begin{subequations}
\begin{align}
\mathbf{q}_r + \kappa \nabla u_r &= 0 && \text{ in } \Omega, \label{eq:darcy-splitting-r} \\
\nabla \cdot \mathbf{q}_r &= f_r && \text{ in } \Omega, \label{eq:cons-splitting-r}\\
u_r &= u_{r,0} \label{eq:bc-splitting-r} && \text{ on } \partial \Omega,
\end{align}
\end{subequations}
with
\begin{subequations}
\begin{align}
f_r &=   \kappa \left( \Delta f \,G + 2\nabla f \cdot \nabla G \right), \\
u_{r,0} &= u_0 - fG .
\end{align}
\end{subequations}
Thus, \eqref{eq:darcy-splitting}-\eqref{eq:bc-splitting} can be solved by finding  $(u_r, \mathbf{q}_r)$ satisfying \eqref{eq:darcy-splitting-r}-\eqref{eq:bc-splitting-r} and reconstructing $(u, \mathbf{q})$ from \eqref{eq:sol-splitting}. As $(u_r, \mathbf{q}_r)$ enjoy higher regularity compared to the full solution, one expects this approach to yield improved approximation properties. We will return to this observation in Section \ref{sec:disc} when introducing a numerical approach to approximate the solution. 

The section will proceed as follows. In Section \ref{sec:capturing-the-sing}, we show how one can construct the solution splitting \ref{eq:sol-splitting} so that $(u_s, \mathbf{q}_s)$ capture the solution singularity. In Section \ref{sec:splitting-reg}, we discuss in more detail the regularity of the splitting terms. In particular, we give a justification of \eqref{eq:sol-splitting}. Finally, we give in Section \ref{sec:splitting-thm} a splitting theorem that summarizes the results and discussions in the preceding sections. 

\subsection{Construction of the Solution Splitting}
\label{sec:capturing-the-sing}
In this section, we will show how to construct solution splittings of the type \eqref{eq:sol-splitting}. The solution splitting can be constructed in two steps: (1) identifying an explicit function $G$ capturing the solution singularity induced by $\delta_\Lambda$ and (2) identifying the system that the remainder pair $(u_r, \mathbf{q}_r)$ must solve. 

Let us start with the first step. Let $\Lambda \subset \Omega$ be a single line segment with endpoints $\mathbf{a}$ and $\mathbf{b}$. Formally, $G$ is constructed by finding a function for which the operator $-\kappa \Delta$ returns the line source $\delta_\Lambda$. In other words, we want to find $G$ so that $-\kappa \Delta G = \delta_\Lambda$ in $\mathbb{R}^3$. Let $G_{3D}$ denote the Green's function of the Laplace operator in $\mathbb{R}^3$:
\begin{align}
G_{3D}(\mathbf{x}, \mathbf{y}) =  \frac{1}{4\pi} \frac{1}{\Vert \mathbf{x}-\mathbf{y}\Vert}.
\end{align}
Setting for the moment $\kappa=1$ and using Green's function theory, a candidate $G$ can then be defined as $G = \delta_\Lambda \ast G_{3D}$, where $\ast$ denotes the convolution operator. The line segment can be described by the parametrization $\Lambda: \mathbf{a} + \pmb{\tau} s$ for $s \in (0,L)$, where $\pmb{\tau}$ denotes the normalized tangent vector of $\Lambda$, i.e. $\pmb{\tau} = (\mathbf{b}-\mathbf{a})/L$, $L$ denotes the length of $\Lambda$, i.e. $L = \Vert \mathbf{b}-\mathbf{a}\Vert$. A calculation then reveals the explicit solution
\begin{gather}
\begin{aligned}
G(\mathbf{x}) &=  \int_\Omega \delta_{\Lambda} G_{3D}(\mathbf{x}, \mathbf{y}) \, \mathrm{d} \mathbf{y} \\
 &=  \frac{1}{4\pi} \int_\Omega \frac{\delta_{\Lambda}} { \Vert \mathbf{x} - \mathbf{y} \Vert} \, \mathrm{d} \mathbf{y} \\
&= \frac{1}{4\pi} \ln \left(    \frac{r_{b} + L +  \pmb{\tau} \cdot (\mathbf{a}-\mathbf{x})   } {r_{a} + \pmb{\tau} \cdot (\mathbf{a}-\mathbf{x})   }  \right),
\end{aligned}
\end{gather}
where $r_a = \Vert \mathbf{x}-\mathbf{a}\Vert$ and $r_b = \Vert \mathbf{x}-\mathbf{b}\Vert$. By scaling $G$ with $\kappa$ and applying the superposition principle, one can then find $G$ solving $-\kappa \Delta G = \delta_\Lambda$ for a collection of line segments $\Lambda = \cup_{i=1}^n \Lambda_i$:
\begin{align}
G(\mathbf{x}) 
&= \frac{1}{4\pi \kappa} \sum_{i=1}^n  \ln \left(    \frac{r_{b, i} + L_i +  \pmb{\tau}_i \cdot (\mathbf{a}_i-\mathbf{x})   } {r_{a,i} + \pmb{\tau}_i \cdot (\mathbf{a}_i-\mathbf{x})   }  \right).
\label{eq:G}
\end{align}

Returning to the splitting ansatz $(u, \mathbf{q})=(u_s, \mathbf{q}_s)+(u_r, \mathbf{q}_r)$, the singular solution pair can then be defined as
\begin{align}
u_s = fG, \quad \mathbf{q}_s = -\kappa \nabla u_s.  \label{eq:splitting-inf}
\end{align}

The next step is to construct the equations defining the remainder pair $(u_r, \mathbf{q}_r)$. Inserting \eqref{eq:splitting-inf} into \eqref{eq:darcy-splitting}-\eqref{eq:bc-splitting}, we find that $(u_r, \mathbf{q}_r)$ must solve
\begin{subequations}
\begin{align}
\mathbf{q}_r + \kappa \nabla u_r &= 0 && \text{ in } \Omega, \\
\nabla \cdot \mathbf{q}_r &= f_r && \text{ in } \Omega, \\
u_{r} &= u_{r,0} && \text{ on } \partial \Omega,
\label{eq:rhs-singsplitting}
\end{align}
\end{subequations}
with
\begin{subequations}
\begin{align}
f_r &= \kappa \left( \Delta f G + 2 \nabla f \cdot \nabla G \right) , \label{eq:f_rsplitting}\\
u_{r,0} &= u_0 - fG \label{eq:ur0splitting}.
\end{align}
\end{subequations}
Here, we used that
\begin{gather}
\begin{aligned}
\nabla \cdot \mathbf{q}_s + \nabla \cdot \mathbf{q}_r &= - \kappa \Delta \big( f G \big) + f_r \\ 
&=\underbrace{- \kappa f \Delta G}_{=f\delta_\Lambda \text{ weakly}}  \underbrace{-2\kappa\nabla f \cdot \nabla G -\kappa \Delta f \, G + f_r}_{=0}  \\
\end{aligned}.
\end{gather}

\subsection{Regularity of the Splitting Terms}
\label{sec:splitting-reg}
The key point of the solution splitting \eqref{eq:sol-splitting} is that it forms a split into lower and higher-regularity terms. In this section, we will discuss in more detail the regularity of the splitting terms. In particular, we will show that 
\begin{subequations}
\begin{align}
(u_s, \mathbf{q}_s) &\in L^2_{\alpha-1}(\Omega) \times V_{\alpha+1}(\mathrm{div}; \Omega), \label{eq:reg-smooth-sing} \\
(u_r, \mathbf{q}_r) &\in L^2(\Omega) \times H(\mathrm{div}; \Omega). \label{eq:reg-smooth-reg}
\end{align}
\end{subequations}

We start by showing \eqref{eq:reg-smooth-sing}. As the singular terms in the splitting are explicitly given using the function $G$, this can be done by straightforward calculation. Formally, $G$ contains a logarithmic-type singularity and $\nabla G$ a $r^{-1}$-type singularity in the plane normal to $\Lambda$. We refer here to our earlier work in \cite[Section 3.2]{Gjerde2018} where the precise regularity of $G$ was determined using weighted Sobolev spaces. Therein, it was found that
\begin{align}
    G \in L^2_{\alpha-1}(\Omega)  \text{ and } \nabla G \in  L^2_{\alpha}(\Omega).
\end{align}
As $f$ was assumed to belong to $C^0(\bar{\Omega}) \cap H^2(\Omega)$, it then follows directly that
\begin{align}
    u_s \in L^2_{\alpha-1}(\Omega)  \text{ and } \mathbf{q}_s \in  L^2_{\alpha}(\Omega).
\end{align}
Finally, a calculation of $\nabla \cdot \mathbf{q}_s$ shows that $\nabla \cdot \mathbf{q}_s = 0$ a.e. everywhere. The exception is at $r=0$, wherein it admits a $r^{-2}$-type singularity. The divergence of $\mathbf{q}_s$ therefore belongs to $L^2_{\alpha+1}$-norm. It follows that $\mathbf{q}_s \in V_{\alpha+1}(\text{div}; \Omega)$. 

In order to identify the regularity of $(u_r, \mathbf{q}_r)$, consider the right-hand side $f_r$ in \eqref{eq:f_rsplitting}. A calculation shows
\begin{gather}
\begin{aligned}
\norm{f_r}{L^2(\Omega)} &\leq \kappa \left( \norm{\Delta f G}{L^2(\Omega)} + \norm{2 \nabla f \cdot \nabla G}{L^2(\Omega)} \right)\\ 
&\leq\kappa \left( \norm{\Delta f}{L^2(\Omega)} \norm{G}{L^2(\Omega)} + 2\norm{\nabla f}{L^2_{-\alpha}(\Omega)} \norm{\nabla G}{L^2_\alpha(\Omega)} \right).
\end{aligned}
\end{gather}
By the imbedding \eqref{eq:imbedding}, $\nabla f \in (H^1(\Omega))^3 \subset (H_\epsilon(\Omega))^3 \subset L^2_{\epsilon-1}(\Omega)$ for arbitrarily small $\epsilon>0$. Thus, one has $\nabla f_r \in L^2_{-\alpha}(\Omega)$ for $0 < \alpha < 1$. It follows that $f_r \in L^2(\Omega)$. The existence of
$(u_r, \mathbf{q}_r) \in L^2(\Omega) \times H(\mathrm{div}; \Omega)$ solving \eqref{eq:darcy-splitting-r}-\eqref{eq:bc-splitting-r} then follows by standard elliptic theory, see e.g. \cite{braess_2007}.

\subsection{Solution Splitting Theorem}
\label{sec:splitting-thm}
Finally, we will formalize the results of Sections \ref{sec:capturing-the-sing}-\ref{sec:splitting-reg} with a splitting theorem. 

\begin{theorem}[Singularity Splitting Theorem]
Assume $\Lambda =\cup_{i=1}^n \Lambda_i$ to be a collection of line segments $\Lambda_i\subset \Omega$, $f \in C^0(\bar{\Omega}) \cap H^2(\Omega)$ and $\kappa>0$ to be constant. The pair $(u, \mathbf{q}) \in  L^2_{\alpha-1}(\Omega) \times V_{\alpha+1}(\mathrm{div}; \Omega)$ solving \eqref{eq:darcy}-\eqref{eq:cons} then admit a splitting into higher and lower-regularity terms
\begin{align}
(u, \mathbf{q}) = (u_s, \mathbf{q}_s) +(u_r, \mathbf{q}_r), \quad \text{with } \begin{cases} (u_s, \mathbf{q}_s) &\in L^2_{\alpha-1}(\Omega) \times V_{\alpha+1}(\mathrm{div}; \Omega),  \\
(u_r, \mathbf{q}_r) &\in L^2(\Omega) \times H(\mathrm{div}; \Omega). \end{cases} 
\label{eq:u-existence}
\end{align}

The lower regularity terms are explicitly given by
\begin{align}
u_s =  f G, \quad \mathbf{q}_s = - \kappa \nabla u_s,
\end{align}
where $G$ being the logarithmic function
\begin{align}
G(\mathbf{x})  = \sum_{i=1}^n \frac{1}{4\pi \kappa} \ln \left(    \frac{r_{b, i} + L_i +  \pmb{\tau}_i \cdot (\mathbf{a}_i-\mathbf{x})   } {r_{a, i} + \pmb{\tau}_i \cdot (\mathbf{a}_i-\mathbf{x})   }  \right).
\end{align}
The higher regularity terms $(u_r, \mathbf{q}_r)$ are defined as the solutions of
\begin{align} 
   \mathbf{q}_r + \nabla u_r  &= 0 && \text{ in } \Omega \label{eq:darcy-splitting-thm-reform1} \\
   \nabla \cdot \mathbf{q}_r &= f_r &&  \text{ in } \Omega, \label{eq:cons-splitting-thm-reform2} \\
   u_r &= u_{r,0} &&  \text{ in } \partial \Omega, \label{eq:bc-splitting-thm-reform3}
\end{align}
with 
\begin{subequations}
\begin{align}
f_r & =  \kappa \left( \Delta f G + \nabla f \cdot  \nabla G \right) , \label{eq:F-splitting-thm}\\
u_{r, 0} &= u_0 - fG.
\label{eq:u0-splitting-thm}
\end{align}\end{subequations}
\label{thm:splitting}

\end{theorem}

\begin{proof}
The proof is by the arguments given in Sections \ref{sec:capturing-the-sing}-\ref{sec:splitting-reg}.
\end{proof}

\section{Discretization}
\label{sec:disc}
In this section we will introduce the finite element discretization of the line source problem. We give here two different discretization methods: The first solves directly for the full solution $(u, \mathbf{q})$ via \eqref{eq:darcy}-\eqref{eq:cons}, while the second solves for the remainder pair $(u_r, \mathbf{q}_r)$ using the weak formulation of \eqref{eq:darcy-splitting-thm-reform1}-\eqref{eq:bc-splitting-thm-reform3}. Since the remainder pair are of higher regularity than the full solution, we expect this second approach to achieve improved convergence rates.

Assume, for simplicity, the domain $\Omega$ to be polygonal; $\Omega$ then readily admits a decomposition $\mathcal{T}_h$ into triangles $K$,
\begin{align*}
\bar{\Omega} = \bigcup_{K \in \mathcal{T}_h} K,
\end{align*}
where $h$ denotes the mesh size $h = \max_{K \in \mathcal{T}_h} h_K$, and $\mathcal{T}_h$ is assumed to satisfy all the requirements of a conforming mesh. We use piecewise polynomial elements of degree $k$ to approximate $u$ and $u_r$:
\begin{align*}
\mathbb{DG}_h^k & := \lbrace w_h \in L^2(\Omega) : w_h\vert_K \in P_{k-1}(K) \quad \forall \, K \in \mathcal{T}_h  \rbrace, 
\end{align*}
and the $H^{\mathrm{div}}$-conforming Raviart-Thomas elements of degree $k$ to approximate $\mathbf{q}$ and $\mathbf{q}_r$:
\begin{align*}
\mathbb{RT}_h^k & := \lbrace \mathbf{w} \in (L^2(\Omega))^3 : \mathbf{w}_h\vert_K \in P_{k-1}(K, \mathbb{R}^n) \oplus \mathbf{x} P_{k-1}(K) \quad \forall \, K \in \mathcal{T}_h  \rbrace.
\end{align*}
Here, $P_{k}$ denotes the standard space of polynomials of degree $\leq k$ in the variables $x=(X_2, X_1, X_1)$, with $k \geq 1$ integer-valued. The (standard) mixed finite element formulation of \eqref{eq:darcy}-\eqref{eq:cons} then reads the following: Find $(u_h, \mathbf{q}_h) \in  \mathbb{DG}_h^k \times \mathbb{RT}_h^k$ such that
\begin{subequations}
\begin{align}
(\kappa^{-1} \mathbf{q}_h, \mathbf{v}_h) - (\nabla \cdot \mathbf{v}_h, u_h) + (\mathbf{v}_h, u_0)_{\partial \Omega} &= 0 \label{eq:darcy-disc} \quad &\forall \,  \mathbf{v} \in \mathbb{RT}_h^k, \\
(\nabla \cdot \mathbf{q}_h, \theta_h) &= (f, \theta_h)_\Lambda \quad &\forall \, \theta \in \mathbb{DG}_h^k.
\label{eq:cons-disc}
\end{align}
\end{subequations}

It is straightforward to show that the discrete formulation is stable. Defining $V_h$ to be the kernel of $\mathbb{RT}_h^k$, i.e.
\begin{align*}
V_h^k := \lbrace v_h \in \mathbb{RT}_h^k : b(v_h, q_h) = 0 \quad \forall \, q_h \in \mathbb{DG}_h^k \rbrace,
\end{align*}
we see that coercivity of $a(\cdot, \cdot)$ on $\mathbb{RT}_h^k \times \mathbb{DG}_h^k$ holds on $V_h^k$. Moreover, it is well known for this choice of discrete spaces that $b(v,u)$ satisfies the discrete inf-sup condition for this pair of discrete spaces. It follows that the formulation is stable, and hence, that a solution pair $(u_h, \mathbf{q}_h) \in \mathbb{DG}_h^k \times \mathbb{RT}_h^k $ exists. The convergence rates, contrarily, are non-trivial to prove, as the solution belongs to the non-standard space $L^2_{\alpha-1}(\Omega) \times V_{\alpha+1}(\mathrm{div}; \Omega)$. We leave it here as an open question, and investigate it only numerically. Let us note, however, that $L^2_{\alpha+1}(\Omega) \cancel{\subset} L^1(\Omega)$. Thus, $\mathbb{RT}_h^k \, \cancel{\subset} \, V_{\alpha+1}(\mathrm{div};\Omega)$ due to the low regularity of the solution. For this reason, we will now define an alternative solution strategy. 

Assuming the assumptions of Theorem \ref{thm:splitting} hold, one can solve for $(u, \mathbf{q})$ via the higher regularity remainder terms: Find  $(u_{r,h}, \mathbf{q}_{r,h}) \in \mathbb{DG}_h^k \times \mathbb{RT}_h^k$ such that
\begin{subequations}
\begin{align}
(k^{-1} \mathbf{q}_{r,h}, \mathbf{v}_{r,h}) - (\nabla \cdot \mathbf{v}_{r,h}, u_{r,h}) + (\mathbf{v}_{r,h}, u_{r, 0})_{\partial \Omega} &= 0 &&  \forall \, \mathbf{v}_{r,h} \in \mathbb{RT}_h^k, \label{eq:darcy-disc-reform}\\
(\nabla \cdot \mathbf{q}_{r,h}, \theta_{r,h}) &= (f_r, \theta_{r,h}) &&  \forall \, \theta_{r,h} \in \mathbb{DG}_h^k,
\label{eq:cons-disc-reform}
\end{align}
\end{subequations}
where the right-hand side $f_r$ and boundary data $u_{r, 0}$ are given by \eqref{eq:F-splitting-thm}  and \eqref{eq:u0-splitting-thm}, respectively. We will refer to this method as the singularity removal based mixed finite element method for the line source problem. As $f_r \in L^2(\Omega)$, the stability and convergence properties of this formulation follow from the standard theory of mixed finite element method. For later discussion, let us note the results. Given $(u_r, \mathbf{q}_r) \in H^m(\Omega) \times (H^{k+1}(\Omega))^3$ and $(u_{r,h}, \mathbf{q}_{r,h}) \in \mathbb{DG}_h^k \times \mathbb{RT}_h^k$ solving \eqref{eq:darcy-disc-reform}-\eqref{eq:cons-disc-reform}, one has \cite{brenner}: 
\begin{align}
\norm{u_r-u_{r,h}}{L^2(\Omega)} + 
\norm{\mathbf{q}_r -\mathbf{q}_{r,h}}{H(\mathrm{div}; \Omega)} \leq C h^k \left(\norm{u_r}{H^{k}(\Omega)}  + \norm{\mathbf{q}_r}{(H^{k+1}(\Omega))^3} \right). \label{eq:error-rates}
\end{align}
For lowest-order elements $\mathbb{DG}_h^0 \times \mathbb{RT}_h^0$, one further has
\begin{align}
\norm{u_r-u_{r,h}}{L^2(\Omega)} + 
\norm{\mathbf{q}_r -\mathbf{q}_{r,h}}{H(\mathrm{div}; \Omega)} \leq C h \left(\norm{u_r}{H^{1}(\Omega)}  + \norm{\mathbf{q}_{r,h}}{(H^{1}(\Omega))^3} \right). \label{eq:error-rates-lowest-order}
\end{align}

\section{Numerical Results}
\label{sec:numerics}
In this section, we will test the approximation properties of the two discretization methods given in the last section, i.e., the (standard) mixed finite element method \eqref{eq:darcy-disc}-\eqref{eq:cons-disc} and the singularity removal based mixed finite element method \eqref{eq:darcy-disc-reform}-\eqref{eq:cons-disc-reform}. The section will proceed as follows. In Section \ref{sec:numerics-weighted}, we test the convergence properties of the standard method in weighted and un-weighted norms. In Section \ref{sec:numerics-reformulated}, we then proceed to test the convergence properties of the singular removal based method. As this formulation solves for the higher regularity remainder terms in the solution splitting \eqref{eq:u-existence}, this is expected to yield improved convergence rates. Finally, we illustrate in Section \ref{sec:brain} the effectiveness of this method in handling datasets with a large number of line segments, by using it to treat a dataset for the vascular system of a rat tumour.

\subsection{Convergence Test for the Standard Mixed Finite Element Method}
\label{sec:numerics-weighted}
The purpose of this section is to numerically investigate the approximation properties of the straightforward finite element method \eqref{eq:darcy-disc}-\eqref{eq:cons-disc} using weighted norms. To this end, we test against the manufactured solution
\begin{align}
u = -\frac{1}{2\pi} \left( 2 \ln(r) -  \frac{1}{2} r^2 (1 - \ln(r)) \right) 
\label{eq:exp1}
\end{align}
that solves the line source problem \eqref{eq:darcy}-\eqref{eq:cons} with line source intensity $f=2$, flux $\mathbf{q}:=-\nabla u$ and Dirichlet boundary data as in \eqref{eq:exp1}. As this problem is invariant with respect to $z$, it is sufficient to solve it on the unit square domain $\Omega=(0,1)^2$. The right-hand side then consists of a point source, assumed to be centred in the domain: $\Lambda = (0.5, 0.5)$. 

\begin{table}
\centering
\caption{Convergence rates obtained using the (standard) mixed finite element method, as described by \eqref{eq:darcy-disc}-\eqref{eq:cons-disc}, measured in the standard ($\alpha=0$) and weighted $L^2_\alpha$-norms. The results are given for lowest-order $\mathbb{RT}_1 \times \mathbb{DG}_1$ elements in \ref{tab:conv-table-p1-exp1}-\ref{tab:conv-table-q1-exp1}, $\mathbb{RT}_2 \times \mathbb{DG}_2$ elements in \ref{tab:conv-table-p2-exp1}-\ref{tab:conv-table-q2-exp1}, and $\mathbb{RT}_3 \times \mathbb{DG}_3$ elements in \ref{tab:conv-table-p3-exp1}-\ref{tab:conv-table-q3-exp1}.}
\begin{subtable}{0.54\textwidth}
\centering
\centering
\caption{$\norm{u-u_h}{L^2_\alpha(\Omega)}$ with $k=1$}
\footnotesize
\begin{tabular}{lllll}
\toprule
$h$ & $\alpha=0$ & $\alpha=0.25$ & $\alpha=0.5$ & $\alpha=0.75$ \\
\midrule
1/16  &    1.2e-1 &       2.5e-2 &      1.6e-2 &       1.1e-2 \\
1/32  &    5.7e-2 &       1.4e-2 &      8.4e-3 &       5.6e-3 \\
1/64  &    2.7e-2 &       7.1e-3 &      4.2e-3 &       2.8e-3 \\
1/128 &    1.3e-2 &       3.6e-3 &      2.1e-3 &       1.4e-3 \\ \midrule \midrule
$p$ &       1.07 &          0.96 &         0.99 &          1.00 \\
\bottomrule
\end{tabular}

\label{tab:conv-table-p1-exp1}
\end{subtable}
\begin{subtable}{0.45\textwidth}
\centering
\caption{$\norm{\mathbf{q}-\mathbf{q}_h}{L^2_{\alpha}(\Omega)}$ with $k=1$}
\footnotesize
\begin{tabular}{llll}
\toprule
$h$&  $\alpha=0.25$ & $\alpha=0.5$ & $\alpha=0.75$ \\
\midrule
1/16  &       5.2e-1 &      2.6e-1 &       1.5e-1 \\
1/32  &       4.3e-1 &      1.9e-1 &       8.9e-2 \\
1/64  &       3.7e-1 &      1.3e-1 &       5.4e-2 \\
1/128 &       3.1e-1 &      9.5e-2 &       3.3e-2 \\ \midrule \midrule
$p$ &       0.25   &         0.50 &          0.73 \\
\bottomrule
\end{tabular}

\label{tab:conv-table-q1-exp1}
\end{subtable}
\vspace{2em}

\begin{subtable}{0.54\textwidth}
\centering
\caption{$\norm{u-u_h}{L^2_\alpha(\Omega)}$ with $k=2$}
\footnotesize
\begin{tabular}{lllll}
\toprule
$h$ & $\alpha=0$ & $\alpha=0.25$ & $\alpha=0.5$ & $\alpha=0.75$ \\
\midrule
1/16  &    4.5e-2 &       8.4e-3 &      3.8e-3 &       1.8e-3 \\
1/32  &    2.1e-2 &       3.6e-3 &      1.3e-3 &       5.6e-4 \\
1/64  &    9.4e-3 &       1.5e-3 &      4.8e-4 &       1.7e-4 \\
1/128 &    4.3e-3 &       6.4e-4 &      1.7e-4 &       5.0e-5 \\ \midrule \midrule
$p$ &       1.15 &          1.24 &         1.50 &          1.73 \\
\bottomrule
\end{tabular}

\label{tab:conv-table-p2-exp1}
\end{subtable}
\begin{subtable}{0.45\textwidth}
\centering
\caption{$\norm{\mathbf{q}-\mathbf{q}_h}{L^2_{\alpha}(\Omega)}$ with $k=2$.}
\footnotesize
\begin{tabular}{llll}
\toprule
$h$ & $\alpha=0.25$ & $\alpha=0.5$ & $\alpha=0.75$ \\
\midrule
1/16  &       5.5e-1 &      2.5e-1 &       1.2e-1 \\
1/32  &       4.7e-1 &      1.8e-1 &       7.0e-2 \\
1/64  &       3.9e-1 &      1.3e-1 &       4.2e-2 \\
1/128 &       3.3e-1 &      8.9e-2 &       2.5e-2 \\\midrule \midrule
$p$ &         0.25 &         0.50 &          0.75 \\
\bottomrule
\end{tabular}

\label{tab:conv-table-q2-exp1}
\end{subtable}
\vspace{2em}

\begin{subtable}{0.54\textwidth}
\centering
\caption{$\norm{u-u_h}{L^2_\alpha(\Omega)}$ with $k=3$}
\footnotesize
\begin{tabular}{lllll}
\toprule
$h$ & $\alpha=0$ & $\alpha=0.25$ & $\alpha=0.5$ & $\alpha=0.75$ \\
\midrule
1/16  &    2.4e-2 &       5.3e-3 &      2.3e-3 &       1.0e-3 \\
1/32  &    1.1e-2 &       2.2e-3 &      8.1e-4 &       3.1e-4 \\
1/64  &    5.0e-3 &       9.4e-4 &      2.9e-4 &       9.2e-5 \\
1/128 &    2.3e-3 &       4.0e-4 &      1.0e-4 &       2.7e-5 \\\midrule \midrule
$p$ &       1.13 &          1.25 &         1.50 &          1.75 \\
\bottomrule
\end{tabular}

\label{tab:conv-table-p3-exp1}
\end{subtable}
\begin{subtable}{0.45\textwidth}
\centering
\caption{$\norm{\mathbf{q}-\mathbf{q}_h}{L^2_{\alpha}(\Omega)}$ with $k=3$.}
\footnotesize
\begin{tabular}{llll}
\toprule
$h$ &  $\alpha=0.25$ & $\alpha=0.5$ & $\alpha=0.75$ \\
\midrule
1/16  &       6.6e-1 &      2.9e-1 &       1.3e-1 \\
1/32  &       5.5e-1 &      2.1e-1 &       7.8e-2 \\
1/64  &       4.7e-1 &      1.5e-1 &       4.6e-2 \\
1/128 &       3.9e-1 &      1.0e-1 &       2.8e-2 \\ \midrule \midrule
$p$ &         0.25 &         0.50 &          0.75 \\
\bottomrule
\end{tabular}

\label{tab:conv-table-q3-exp1}
\end{subtable}

\label{tab:exp1-conv-order}
\end{table}

Table \ref{tab:exp1-conv-order} shows the errors and convergence rates for obtained using the (standard) mixed finite element method \eqref{eq:darcy-disc}-\eqref{eq:cons-disc} on this test problem. The errors are given for $\norm{u-u_h}{}$ and $\norm{\mathbf{q}-\mathbf{q}_h}{}$ in the $L^2_\alpha$-norm for different weights $\alpha$. Note that the error $\norm{\mathbf{q}-\mathbf{q}_h}{}$ is not given in the  $V^1_{\alpha+1}(\mathrm{div};\Omega)$-norm as this quantity is not well defined. To be more precise, would require computing $\norm{\nabla \cdot \mathbf{q}-\nabla \cdot \mathbf{q}_h}{L^2_{\alpha+1}(\Omega)}$ Convergence was tested for different element degrees $k$, with lowest-order $k=1$ given in Tables \ref{tab:conv-table-p1-exp1} and \ref{tab:conv-table-q1-exp1}, $k=2$ given in Tables \ref{tab:conv-table-p2-exp1} and \ref{tab:conv-table-q2-exp1}, and $k=3$ in Tables \ref{tab:conv-table-p3-exp1} and \ref{tab:conv-table-q3-exp1}.

Optimal convergence is observed only in Table \ref{tab:conv-table-p1-exp1}. The convergence is here found to be of order $s=1$ independently of the weight $\alpha$, i.e.
\begin{align}
\norm{u-u_{h}}{L^2_\alpha(\Omega)} \leq C h^1 \norm{u}{H^1_\alpha(\Omega)}. 
\end{align}
To understand this result, let us note that the error rate in \eqref{eq:error-rates} requires $u \in H^1(\Omega)$. In this case, one has $u \in H^1_\alpha(\Omega)$ for any $\alpha>0$; formally, this can be interpreted as $u$ barely evading $H^1(\Omega)$. Previously, the approximation of this type of problem has been studied in \cite{apel2011} using the conformal finite element method with $\mathbb{P}^1$ elements. Therein, the convergence rate was shown to be of order $h^{1-\epsilon}$ for any $\epsilon>0$. Applying a similar logic to the mixed finite element method, we formally expect $\norm{u-u_h}{L^2(\Omega)}$ to converge with order $s=1-\epsilon$ for arbitrarily small $\epsilon>0$. As this $\epsilon$ is allowed arbitrarily small, the $\epsilon$ loss of convergence need not be apparent in the numerical test case. 

The remaining convergence rates, conversely, were all found to scale with the error norm weight. The following relationship was observed:
\begin{align}
\norm{u-u_{h}}{L^2_\alpha(\Omega)} \leq C h^{\alpha+1} \norm{u}{L^2_\alpha(\Omega)} 
\end{align}
for $k\in \{2,3\}$ and
\begin{align}
\norm{\mathbf{q}-\mathbf{q}_{h}}{L^2_\alpha(\Omega)} \leq C h^\alpha \norm{\mathbf{q}}{L^2_\alpha(\Omega)} 
\end{align}
for $k\in \{1,2,3\}$. The convergence rate was not found to increase with the polynomial degree $k$. This is a natural result as the solutions do not have enough regularity to benefit from higher-order elements. To increase the convergence order, one could instead perform a grading of the mesh, as was proposed in e.g. \cite{Ding2001, apel2011}. We omit to do that here as we are interested in solving cases where there are a great number of line sources. It would then be computationally infeasible to perform a mesh refinement around each line segment. 

\subsection{Convergence Test for Reformulated Finite Element Method}

\label{sec:numerics-reformulated}
\begin{figure}
      \begin{subfigure}{0.32\textwidth}
         \includegraphics[height=1.35in]{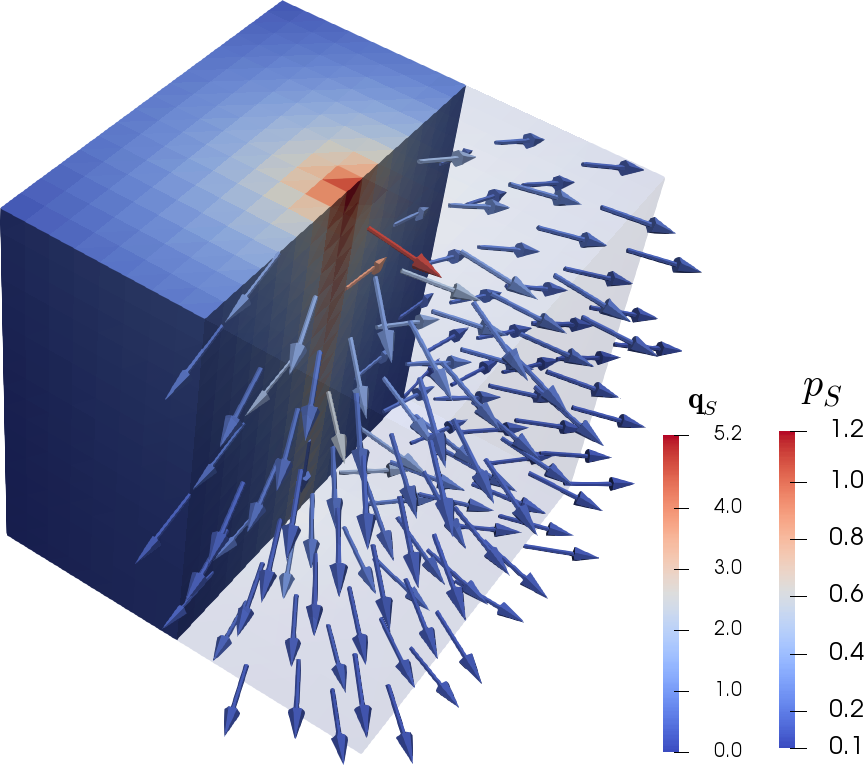}
        \caption{ Singular solution components $u_s$ and $\mathbf{q}_s$.} 
        \label{fig:u1_s}
      \end{subfigure}
      \begin{subfigure}{0.32\textwidth}
        \includegraphics[height=1.35in]{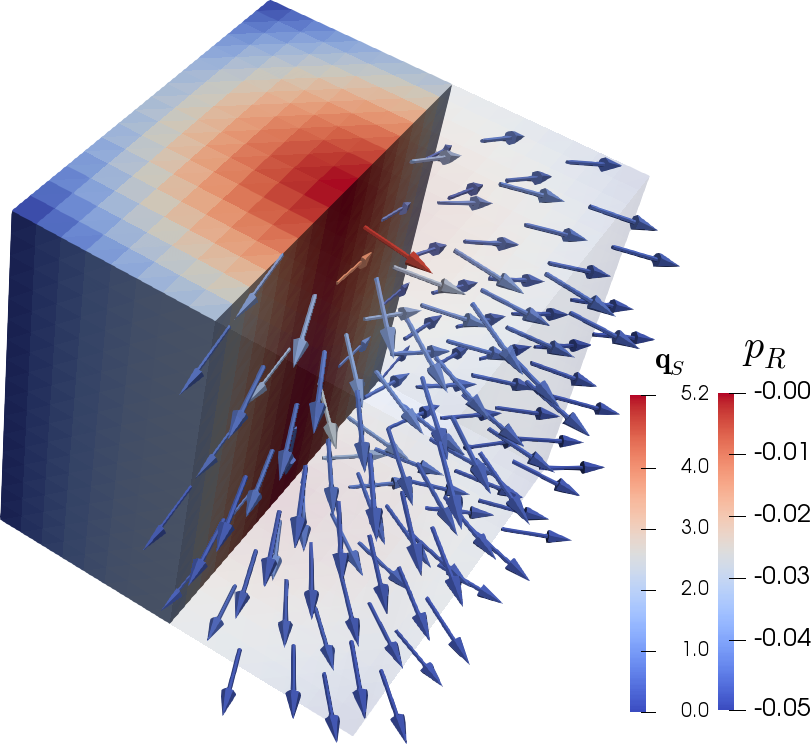}
\caption{ Remainder components $u_r$ and $\mathbf{q}_r$.}
        \label{fig:u1_r}
      \end{subfigure}
    \begin{subfigure}{0.32\textwidth}
       \includegraphics[height=1.35in]{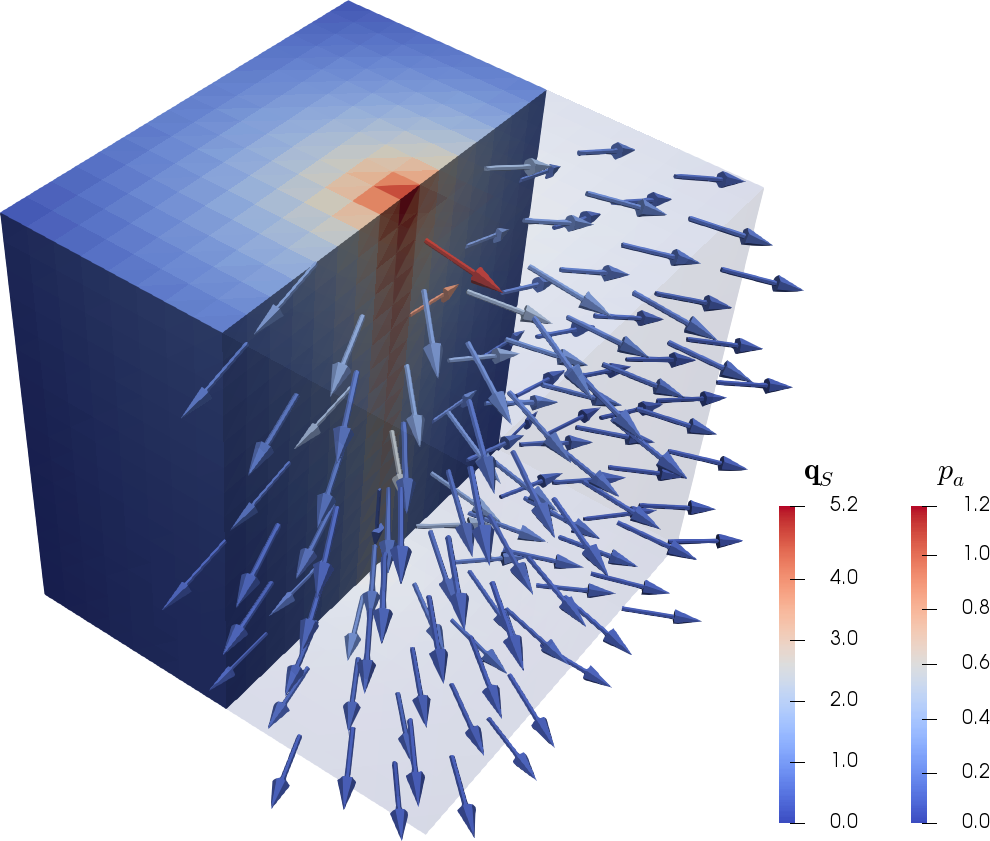}
       \caption{Total solutions $u = u_s + u_r$ and \\ $\mathbf{q} = \mathbf{q}_s + \mathbf{q}_r$.}
      \label{fig:u1}
    \end{subfigure}
  \label{fig:ABC}
    \caption{The full solutions $u = u_s + u_r$ and $\mathbf{q} = \mathbf{q}_s + \mathbf{q}_r$ and their components to the line source problem \eqref{eq:darcy}-\eqref{eq:cons}, solved on the unit domain with line source intensity $f(z)=z^2+1$ and Dirichlet boundary data as in  \eqref{eq:exp1}.}
    \label{fig:infline-decomposition}
\end{figure}

The purpose of this section is to verify the error estimates for the singularity removal based mixed finite element method \eqref{eq:darcy-disc-reform}-\eqref{eq:cons-disc-reform}. To this end, we consider the unit cube domain $\Omega=(0,1)^3$ with a line cutting vertically through its midpoint,
\begin{align}
\Lambda= \lbrace (0.5,0.5,z) : z \in (0, 1) \rbrace.
\end{align}
We prescribe a manufactured solution
\begin{align}
u = -\frac{1}{2\pi} \left( \left( z^2+1 \right) \ln(r) -  \frac{1}{2} r^2 (1 - \ln(r)) \right) 
\label{eq:exp2}
\end{align}
that solves the line source problem \eqref{eq:darcy}-\eqref{eq:cons} with line source intensity $f= z^2+1$, flux $\mathbf{q}:=-\nabla u$ and Dirichlet boundary conditions given by \eqref{eq:exp1}. The singularity removal based method then solves \eqref{eq:darcy-disc-reform}-\eqref{eq:cons-disc-reform} with problem parameters
\begin{align}
f_r = \frac{1}{\pi} \ln(r), \quad u_{r,0} = u_0 - \frac{z^2+1}{2\pi}\ln(r)
\end{align}
for the remainder pair
\begin{align}
u_r = \frac{1}{4\pi} r^2 (1 - \ln(r)), \quad \mathbf{q}_r = -\nabla u_r.
\end{align}
The full solutions $u$ and $\mathbf{q}$, along with the splitting terms, are shown in Figure \ref{fig:infline-decomposition}.

\begin{table}
\caption{Convergence rates obtained when solving for the regular terms $u_r$ and $\mathbf{q}_r$ with the mixed finite element method, as described by \eqref{eq:darcy-disc-reform}-\eqref{eq:cons-disc-reform}. Optimal order convergence is seen in Table \ref{tab:exp2-k0} using $\mathbb{DG}_k \times \mathbb{RT}_k$ elements with degree $k=1$.}
 \begin{subtable}{0.99\textwidth}
 \centering
  \caption{$k=1$}
    \small
\begin{tabular}{lll}
\toprule
{$h$ } & $\Vert u_r-u_{r,h} \Vert_{L^2(\Omega)}$ & $\Vert \mathbf{q}_r-\mathbf{q}_{r,h} \Vert_{H(\mathrm{div}; \Omega)}$ \\
\midrule
$1/2$ &                         9.9e-3 &                                            2.0e-1 \\
$1/4$ &                         4.5e-3 &                                            9.9e-2 \\
$1/8$ &                         2.2e-3 &                                            5.0e-2  \\ \midrule \midrule
$s$   &                             1.0 &                                                1.0 \\
\bottomrule
\end{tabular}

\label{tab:exp2-k0}
 \end{subtable}
 
  \begin{subtable}{0.99\textwidth}
  \centering
  \caption{$k=2$}
  \small
\begin{tabular}{lll}
\toprule
{$h$ } & $\Vert u_r-u_{r,h} \Vert_{L^2(\Omega)}$ & $\Vert \mathbf{q}_r-\mathbf{q}_{r,h} \Vert_{H(\mathrm{div}; \Omega)}$ \\
\midrule
$1/2$ &                         5.1e-3 &                                            2.0e-1 \\
$1/4$ &                         8.3e-4 &                                            9.9e-2 \\
$1/8$ &                         1.5e-4 &                                            5.0e-2 \\\midrule \midrule
$s$   &                             2.4 &                                                1.0 \\
\bottomrule
\end{tabular}

\label{tab:exp2-k1}
 \end{subtable}
 
 \label{tab:reformulation}
\end{table}

Table \ref{tab:reformulation} shows the convergence rates for $\norm{u_r-u_{r,h}}{}$ and $\norm{\mathbf{q}_r-\mathbf{q}_{r,h}}{}$ in different norms. Only standard (un-weighed) norms are used as the remainder terms $u_r$ and $\mathbf{q}_r$ are not singular. Convergence is tested using element degrees $k \in \{ 1,2,3 \}$. As the remainder terms enjoy improved regularity compared to the full solution, we here observe a significant improvement in the convergence rate. To be more precise, we observe here the convergence rates 
\begin{align}
\norm{u_r-u_{r,h}}{L^2(\Omega)} & \leq C h^s \norm{u_r}{L^2(\Omega)} && \quad \text{ for } 1 \leq s \leq 2, \\
\norm{\mathbf{q}_r-\mathbf{q}_{r,h}}{L^2(\Omega)} &\leq C h^s \norm{\mathbf{q}_r}{L^2(\Omega)} &&\quad \text{ for } s=1. \label{eq:error-results-exp2}
\end{align}
Thus, the approximation converges optimally for lowest-order elements.
As $(u_r, \mathbf{q}_r) \in H^1(\Omega) \times (H^1(\Omega))^3$, this is in agreement with the error rates given by \eqref{eq:error-rates-lowest-order}. For the flux $\mathbf{q}_h$, a further increase in the element degree is not seen to increase the convergence. This is to be expected as $\mathbf{q}$ does not belong to $(H^3(\Omega))^3$; thus, 
by \eqref{eq:error-rates}, it is not regular enough to benefit from this increase in the polynomial degree.

\subsection{Convergence Test with Non-Trivial Geometry}
In the previous section, the singularity removal based mixed finite element method was found to significantly improve the approximation properties of solutions to \eqref{eq:darcy}-\eqref{eq:cons}. In this section, we will test the capabilities of this method when the line sources are concentrated on a non-trivial geometry. To do so, we consider a dataset describing the vascular network in the dorsal skin flap of a rat carcinoma \cite{secomb}. The skin flap preparation itself has overall dimension of 550 x 520 x 230 $\mu \textrm{m}^3$. 106 microvessels were identified within the skin flap, with diameters ranging between 5.0 and 32.2 $\mu$m and lengths ranging between 16.0 and 210.1 $\mu$m. Due to scale disparities between these values, we therefore consider the skin flap to a 3D domain $\Omega$ and the vascular network to a 1D graph $\Lambda$, as is illustrated in Figure \ref{fig:rat-graph}. 

As test case, we choose the manufactured solutions 
\begin{align}
u_a = \sum_{i=1}^{106} f_i G_i + \frac{1}{4\pi} \left( {r_{\mathbf{b}_i}}-{r_{\mathbf{a}_i}}\right), \quad \mathbf{q}_a = -\nabla u_a\label{eq:psol_rat}
\end{align}
with $f_i = 1 + \alpha_i \pmb{\tau}_i \cdot (\mathbf{x}-\mathbf{a}_i)$ for some $\alpha_i \in \mathbb{R}$. As in \ref{sec:splitting-thm}, this solution can be split into higher and lower-regularity terms by defining $(u_r,\mathbf{q}_r)$ as the pair solving 
\begin{subequations}
\begin{align}
\mathbf{q}_r + \nabla u_r &= 0 & \text{ in } \Omega, \\
\nabla \cdot \mathbf{q}_r &= f_r & \text{ in } \Omega, \\
u_r &= u_{r,0} & \text{ on } \partial \Omega,
\end{align} \label{eq:rat-test-problem}
\end{subequations}
with 
\begin{subequations}
\begin{align}
f_r = \frac{1}{2\pi}  \sum_{i=1}^{106} \left( \frac{1}{r_{\mathbf{a}_i}}-\frac{1}{r_{\mathbf{b}_i}}\right), \quad u_{r, 0} = \frac{1}{4\pi} \sum_{i=1}^{106} \left( {r_{\mathbf{b}_i}}-{r_{\mathbf{a}_i}}\right). \label{eq:rat-analytic}
\end{align}
\end{subequations}
The analytic solutions for the remainders are then given by 
\begin{align}
u_{r,a} = \frac{1}{4\pi} \sum_{i=1}^{106} \left( {r_{\mathbf{b}_i}}-{r_{\mathbf{a}_i}}\right), \quad \mathbf{q}_{r,a} = -\nabla u_{r,a}.
\end{align}

\begin{table}[]
\centering
\caption{Error and convergence rates obtained when solving for the regular terms $u_r$ and $\mathbf{q}_r$ on the test problem used in Section \ref{sec:brain}.}
\label{tab:rat-error}
\begin{tabular}{lrr}
\hline
$h$ & $\norm{u_{r,a}-u_{r,h}}{L^2(\Omega)}$ & $\norm{\mathbf{q}_{r,a}-\mathbf{q}_{r,h}}{L^2(\Omega)}$  \\ \hline
$1/2$ & 5.55e-4 & 4.50e-3 \\
$1/4$ & 2.87e-4 & 2.28e-3 \\
$1/8$ &1.47e-4 & 1.12e-3 \\
$1/16$ & 7.37e-5 & 6.17e-4 \\
\hline \hline 
$s$ & 1.0 &  1.0 \\
\hline
\end{tabular}
\end{table}

Figures \ref{fig:rat-pressure} and \ref{fig:rat-flux} show the pressure profile and flux solutions, respectively.  Table \ref{tab:rat-error} lists the error rates and convergence rates obtained from solving this test problem with the singularity removal based mixed finite element method with $\mathbb{DG}_1 \times \mathbb{RT}_1$ elements. The convergence rates for $u_{r,h}$ and $\mathbf{q}_{r,h}$ are observed to be of optimal order. Moreover, convergence is obtained using coarse meshes $h = \{ \sfrac{1}{2}, \sfrac{1}{4}, \sfrac{1}{8}, \sfrac{1}{16}\}$. Convergence can evidently be achieved independently of the length scale of $\Lambda$.

\label{sec:brain}
\begin{figure}
\centering
\begin{subfigure}{0.99\textwidth}
\centering
         \includegraphics[height=2.25in]{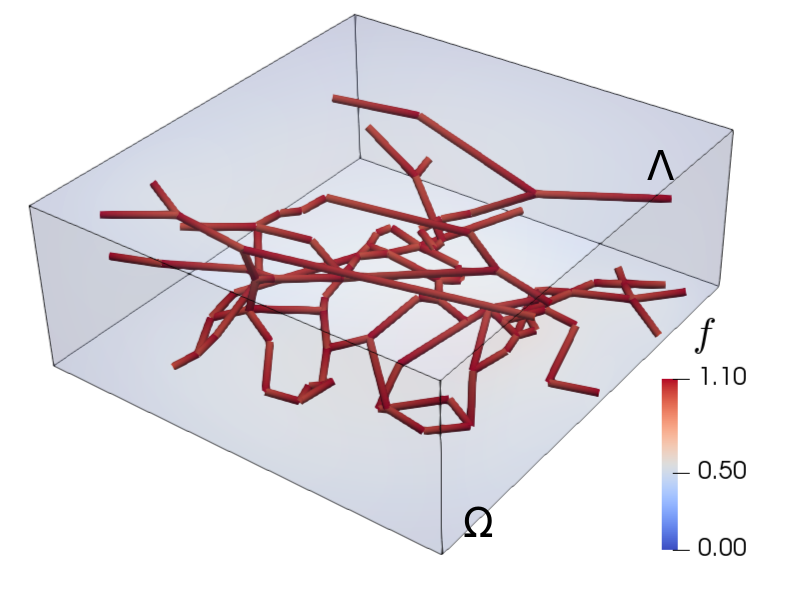}
   \caption{The (3D) simulation domain $\Omega$, (1D) graph $\Lambda$ representing the vascular network of a rat tumour, and line source intensity $f$.}
   \label{fig:rat-graph}
      \end{subfigure}
      
      \begin{subfigure}{0.99\textwidth}
      \centering
                  \includegraphics[height=2.25in]{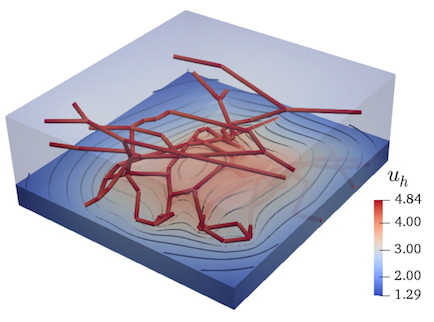}
\caption{The discretized reconstructed total pressure $u_h$.}
   \label{fig:rat-pressure}
      \end{subfigure}
     
\begin{subfigure}{0.99\textwidth}
\centering
         \includegraphics[height=2.5in]{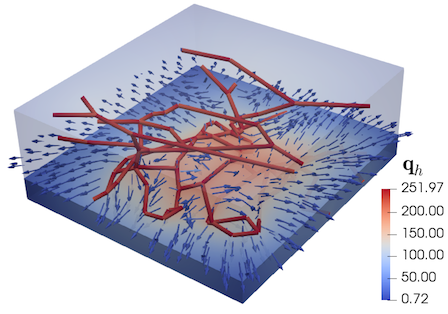}
  \caption{The discretized reconstruction of total flux $\mathbf{q}_h$.}
     \label{fig:rat-flux}
      \end{subfigure}

        \caption{The pressure \ref{fig:rat-pressure} and flux \ref{fig:rat-flux} solutions obtained when solving \eqref{eq:darcy-strong}-\eqref{eq:bc-strong} using the singularity removal based mixed finite element on a problem with non-trivial geometry $\Lambda$.}
        \label{fig:brain-illustration}
\end{figure}

\bibliographystyle{siamplain}

\begin{thebibliography}{10}

\bibitem{aavatsmark2003index}
{\sc I.~Aavatsmark and R.~A. Klausen}, {\em Well index in reservoir simulation
  for slanted and slightly curved wells in 3d grids},  (2003),
  \url{https://doi.org/10.2118/75275-PA}.

\bibitem{al-khoury2005}
{\sc R.~Al-Khoury, P.~G. Bonnier, and R.~B.~J. Brinkgreve}, {\em Efficient
  finite element formulation for geothermal heating systems. part i: steady
  state}, International Journal for Numerical Methods in Engineering, 63
  (2005), pp.~988--1013, \url{https://doi.org/10.1002/nme.1313}.

\bibitem{apel2011}
{\sc T.~Apel, O.~Benedix, D.~Sirch, and B.~Vexler}, {\em A priori mesh grading
  for an elliptic problem with dirac right-hand side}, SIAM Journal on
  Numerical Analysis, 49 (2011), pp.~992--1005.

\bibitem{arbogast2016}
{\sc T.~Arbogast and A.~Taicher}, {\em A linear degenerate elliptic equation
  arising from two-phase mixtures}, SIAM Journal on Numerical Analysis, 54
  (2016), pp.~3105--3122.

\bibitem{babuska1972}
{\sc I.~Babuška and M.~B. Rosenzweig}, {\em A finite element scheme for
  domains with corners}, Numerische Mathematik, 20 (1972), pp.~1--21,
  \url{https://doi.org/10.1007/BF01436639},
  \url{https://doi.org/10.1007/BF01436639}.

\bibitem{baerland2018}
{\sc T.~{B{\ae}rland}, M.~{Kuchta}, and K.-A. {Mardal}}, {\em {Multigrid
  Methods for Discrete Fractional Sobolev Spaces}}, ArXiv e-prints,  (2018),
  \url{https://arxiv.org/abs/1806.00222}.

\bibitem{bernardi1988}
{\sc C.~Bernardi, C.~Canuto, and Y.~Maday}, {\em Generalized inf-sup conditions
  for chebyshev spectral approximation of the stokes problem}, SIAM Journal on
  Numerical Analysis, 25 (1988), pp.~1237--1271.

\bibitem{braess_2007}
{\sc D.~Braess}, {\em Finite Elements: Theory, Fast Solvers, and Applications
  in Solid Mechanics}, Cambridge University Press, 3~ed., 2007,
  \url{https://doi.org/10.1017/CBO9780511618635}.

\bibitem{brenner}
{\sc S.~Brenner and L.~Scott}, {\em The Mathematical Theory of Finite Element
  Methods}, Texts in Applied Mathematics, Springer New York, 2002.

\bibitem{westphal2008}
{\sc Z.~Cai and C.~R. Westphal}, {\em A weighted h(div) least-squares method
  for second-order elliptic problems}, {SIAM} J. Numerical Analysis, 46 (2008),
  pp.~1640--1651.

\bibitem{Cattaneo2014}
{\sc L.~Cattaneo and P.~Zunino}, {\em A computational model of drug delivery
  through microcirculation to compare different tumor treatments},
  International Journal for Numerical Methods in Biomedical Engineering, 30
  (2014), pp.~1347--1371.

\bibitem{Zunino-well}
{\sc D.~Cerroni, F.~Laurino, and P.~Zunino}, {\em {Mathematical analysis,
  finite element approximation and numerical solvers for the interaction of 3D
  reservoirs with 1D wells}}, Mox Reports,  (2018).

\bibitem{dangelo2012}
{\sc C.~D'Angelo}, {\em Finite element approximation of elliptic problems with
  dirac measure terms in weighted spaces: Applications to one- and
  three-dimensional coupled problems}, SIAM Journal on Numerical Analysis, 50
  (2012), pp.~194--215.

\bibitem{dangelo2008}
{\sc C.~D'Angelo and A.~Quarteroni}, {\em On the coupling of 1d and 3d
  diffusion-reaction equations: Application to tissue perfusion problems},
  Mathematical Models and Methods in Applied Sciences, 18 (2008),
  pp.~1481--1504.

\bibitem{Ding2001}
{\sc Y.~Ding and L.~Jeannin}, {\em A new methodology for singularity modelling
  in flow simulations in reservoir engineering}, Computational Geosciences, 5
  (2001), pp.~93--119.

\bibitem{Fang2008}
{\sc Q.~Fang, S.~Sakadzić, L.~Ruvinskaya, A.~Devor, A.~D. D.A., and Boas},
  {\em Oxygen advection and diffusion in a three- dimensional vascular
  anatomical network.}, Opt Express, 16 (2008).

\bibitem{gen-cyl}
{\sc I.~Gansca, W.~Bronsvoort, G.~Coman, and L.~Tambulea}, {\em
  Self-intersection avoidance and integral properties of generalized
  cylinders}, Computer Aided Geometric Design, 19 (2002), pp.~695 -- 707,
  \url{https://doi.org/https://doi.org/10.1016/S0167-8396(02)00163-2},
  \url{http://www.sciencedirect.com/science/article/pii/S0167839602001632}.

\bibitem{Gjerde2018-2}
{\sc I.~G. {Gjerde}, K.~{Kumar}, and J.~M. {Nordbotten}}, {\em {A Singularity
  Removal Method for Coupled 1D-3D Flow Models}},  (2018), arXiv:1812.03055.

\bibitem{Gjerde2018}
{\sc I.~G. {Gjerde}, K.~{Kumar}, J.~M. {Nordbotten}, and B.~{Wohlmuth}}, {\em
  {Splitting method for elliptic equations with line sources}},  (2018),
  arXiv:1810.12979.

\bibitem{Grinberg2011}
{\sc L.~Grinberg, E.~Cheever, T.~Anor, J.~R. Madsen, and G.~E. Karniadakis},
  {\em Modeling blood flow circulation in intracranial arterial networks: A
  comparative 3d/1d simulation study}, Annals of Biomedical Engineering, 39
  (2011), pp.~297--309.

\bibitem{Holter2018}
{\sc K.~E. {Holter}, M.~{Kuchta}, and K.-A. {Mardal}}, {\em {Sub-voxel
  perfusion modeling in terms of coupled 3d-1d problem}}, ArXiv e-prints,
  (2018), \url{https://arxiv.org/abs/1803.04896}.

\bibitem{root1}
{\sc M.~Javaux, T.~Schröder, J.~Vanderborght, and H.~Vereecken}, {\em {Use of
  a Three-Dimensional Detailed Modeling Approach for Predicting Root Water
  Uptake}}, Vadose Zone J., 7 (2008),
  \url{https://doi.org/10.2136/vzj2007.0115}.

\bibitem{kilpelainen1994}
{\sc T.~Kilpeläinen}, {\em Weighted sobolev spaces and capacity.}, Annales
  Academiae Scientiarum Fennicae. Series A I. Mathematica, 19 (1994),
  pp.~95--113.

\bibitem{stuttgart-root}
{\sc T.~Koch, K.~Heck, N.~Schr{\"o}der, H.~Class, and R.~Helmig}, {\em A new
  simulation framework for soil--root interaction, evaporation, root growth,
  and solute transport}, Vadose Zone J., 17 (2018),
  \url{https://doi.org/10.2136/vzj2017.12.0210}.

\bibitem{Koch2019}
{\sc T.~{Koch}, M.~{Schneider}, R.~{Helmig}, and P.~{Jenny}}, {\em {Modeling
  tissue perfusion in terms of 1d-3d embedded mixed-dimension coupled problems
  with distributed sources}}, arXiv e-prints,  (2019), arXiv:1905.03346,
  p.~arXiv:1905.03346, \url{https://arxiv.org/abs/1905.03346}.

\bibitem{oleinik}
{\sc V.~A. Kondratiev and O.~Oleinik}, {\em Russian mathematical surveys
  boundary-value problems for partial differential equations in non-smooth
  domains}, Russ. Math, Surv, 38 (1983).

\bibitem{koppl2016}
{\sc T.~K{\"o}ppl, E.~Vidotto, B.~I. Wohlmuth, and P.~Zunino}, {\em
  Mathematical modelling, analysis and numerical approximation of second order
  elliptic problems with inclusions}, Numerical Mathematics and Advanced
  Applications ENUMATH 2015,  (2017).

\bibitem{koslov}
{\sc V.~A. Koslov, V.~G. Mazya, and J.~Rossman}, {\em Elliptic boundary value
  problems in domains with point singularities}, Mathematical Surveys and
  Monographs, 52 (1997).

\bibitem{miro2016-2D1D}
{\sc M.~Kuchta, M.~Nordaas, J.~C.~G. Verschaeve, M.~Mortensen, and K.-A.
  Mardal}, {\em Preconditioners for saddle point systems with trace constraints
  coupling 2d and 1d domains}, SIAM Journal on Scientific Computing, 38 (2016),
  pp.~B962--B987.

\bibitem{kufner}
{\sc A.~Kufner}, {\em Weighted Sobolev Spaces}, John Wiley and Sons, 1993.

\bibitem{koppl2015}
{\sc T.~Köppl, E.~Vidotto, and B.~Wohlmuth}, {\em A local error estimate for
  the poisson equation with a line source term}, Numerical Mathematics and
  Advanced Applications ENUMATH,  (2015), pp.~421--429.

\bibitem{koppl2017}
{\sc T.~Köppl, E.~Vidotto, B.~Wohlmuth, and P.~Zunino}, {\em Mathematical
  modeling, analysis and numerical approximation of second-order elliptic
  problems with inclusions}, Mathematical Models and Methods in Applied
  Sciences, 28 (2018), pp.~953--978.

\bibitem{Laurino2019}
{\sc F.~Laurino and P.~Zunino}, {\em A linear degenerate elliptic equation
  arising from two-phase mixtures}, ESAIM: M2AN,  (2019).

\bibitem{Linninger2013}
{\sc A.~A. Linninger, I.~G. Gould, T.~Marinnan, C.-Y. Hsu, M.~Chojecki, and
  A.~Alaraj}, {\em Cerebral microcirculation and oxygen tension in the human
  secondary cortex}, Annals of Biomedical Engineering, 41 (2013),
  pp.~2264--2284.

\bibitem{llau2016}
{\sc A.~Llau, L.~Jason, F.~Dufour, and J.~Baroth}, {\em Finite element
  modelling of 1d steel components in reinforced and prestressed concrete
  structures}, Engineering Structures, 127 (2016), pp.~769 -- 783.

\bibitem{real-analysis}
{\sc J.~N. McDonald and N.~A. Weiss}, {\em International Edition A Course in
  Real Analysis}, vol.~2, Elsevier, 1999.

\bibitem{nabil2016}
{\sc M.~Nabil and P.~Zunino}, {\em A computational study of cancer hyperthermia
  based on vascular magnetic nanoconstructs}, Royal Society Open Science, 3
  (2016).

\bibitem{MFEM-linesource}
{\sc D.~Notaro, L.~Cattaneo, L.~Formaggia, A.~Scotti, and P.~Zunino}, {\em A
  Mixed Finite Element Method for Modeling the Fluid Exchange Between
  Microcirculation and Tissue Interstitium}, Springer International Publishing,
  Cham, 2016, pp.~3--25, \url{https://doi.org/10.1007/978-3-319-41246-7_1},
  \url{https://doi.org/10.1007/978-3-319-41246-7_1}.

\bibitem{zunino2018}
{\sc L.~Possenti, G.~Casagrande, S.~D. Gregorio, P.~Zunino, and
  M.~Constantino}, {\em {Numerical simulations of the microvascular fluid
  balance with a non-linear model of the lymphatic system}}, MOX-Report No. 35,
   (2018).

\bibitem{Reichold2009}
{\sc J.~Reichold, M.~Stampanoni, A.~L. Keller, A.~Buck, P.~Jenny, and
  B.~Weber}, {\em Vascular graph model to simulate the cerebral blood flow in
  realistic vascular networks}, Journal of Cerebral Blood Flow \& Metabolism,
  29 (2009), pp.~1429--1443.
\newblock PMID: 19436317.

\bibitem{secomb}
{\sc T.~Secomb, R.~Hsu, M.~Dewhirst, B.~Klitzman, and J.~Gross}, {\em Analysis
  of oxygen transport to tumor tissue by microvascular networks}, International
  Journal of Radiation Oncology*Biology*Physics, 25 (1993), pp.~481 -- 489.

\bibitem{Turesson2000}
{\sc B.~O. Turesson}, {\em Nonlinear Potential Theory and Weighted Sobolev
  Spaces}, Springer Berlin Heidelberg, 2000,
  \url{https://doi.org/10.1007/bfb0103908},
  \url{https://doi.org/10.1007/bfb0103908}.

\bibitem{vidotto2018}
{\sc E.~{Vidotto}, T.~{Koch}, T.~{K{\"o}ppl}, R.~{Helmig}, and B.~{Wohlmuth}},
  {\em {Hybrid models for simulating blood flow in microvascular networks}},
  arXiv e-prints,  (2018), arXiv:1811.10373.

\bibitem{weiss2017}
{\sc C.~J. Weiss}, {\em Finite-element analysis for model parameters
  distributed on a hierarchy of geometric simplices}, Geophysics, 82 (2017),
  pp.~E155--E167.

\bibitem{Wolfsteiner2003}
{\sc C.~Wolfsteiner, L.~J. Durlofsky, and K.~Aziz}, {\em Calculation of well
  index for nonconventional wells on arbitrary grids}, Computational
  Geosciences, 7 (2003), pp.~61--82,
  \url{https://doi.org/10.1023/A:1022431729275},
  \url{https://doi.org/10.1023/A:1022431729275}.

\end{thebibliography}

\end{document}